\documentclass[a4paper,12pt]{article}
\usepackage{amsmath, amsthm, amstext}
\usepackage{amssymb, array, amsfonts}
\usepackage[utf8]{inputenc}
\usepackage{graphicx}

\numberwithin{equation}{section}

\def \al{\alpha}
\def \be{\beta}

\def \de{\delta}
\def \er{\varepsilon}
\def \ze{\zeta}
\def \ka{\varkappa}
\def \la{\lambda}
\def \si{\sigma}

\def \ph{\varphi}
\def \oo{\omega}

\def \G{\Gamma}
\def \D{\Delta}

\def \C{\mathbb{C}}
\def \N{\mathbb{N}}
\def \R{\mathbb{R}}

\def\n{\nabla}
\def\dd{\partial}
\def\div{\operatorname{div}}
\def\rot{\operatorname{rot}}
\def\1{1\!\!\!\!1}
\def\const{\operatorname{const}}

\def\dom{\operatorname{Dom}}

\def\lip{\operatorname{Lip}}

\def\ran{\operatorname{Ran}}

\newcommand{\<}{\langle}
\renewcommand{\>}{\rangle}

\theoremstyle{plain}
\newtheorem{theorem}{\bf Theorem}[section]
\newtheorem{lemma}[theorem]{\bf Lemma}

\newtheorem{cor}[theorem]{\bf Corollary}

\theoremstyle{definition}
\newtheorem{defi}[theorem]{\bf Definition}

\theoremstyle{remark}
\newtheorem{rem}[theorem]{\bf Remark}

\renewcommand{\le}{\leqslant}
\renewcommand{\ge}{\geqslant}

\renewcommand{\qed}{\vrule height7pt width5pt depth0pt}

\title{Maxwell operator in a cylinder.
Separation of variables}

\author{N.~D.~Filonov
\thanks{This work is supported by the project Russian Science Foundation 17-11-01069.}}
\date{}
\begin{document}
\maketitle

\begin{abstract}
The Maxwell operator in a 3D cylinder is considered.
The coefficients are assumed to be scalar functions depending on the longitudinal variable only.
Such operator is represented as a sum of countable set of matrix differential operators of first order
acting in $L_2(\R)$.
Based on this representation we give a detailed description of the structure of the spectrum
of the Maxwell operator in two particular cases:
1) in the case of coefficients stabilizing at infinity; and
2) in the case of periodic coefficients.
\end{abstract}

\section*{Introduction}
Let $U \subset \R^2$ be a bounded connected domain with Lipschitz boundary.
Let $\Pi = U \times \R$ be a three-dimensional cylinder.
We study the self-adjoint Maxwell operator ${\cal M}$ in the cylinder $\Pi$ 
under the boundary conditions of perfect conductivity.
It is a matrix differential operator of first order with coefficients $\er(x)$, $\mu(x)$, 
describing the dielectric and the magnetic permeabilities of the medium filling the cylinder.
The exact definition will be given below, see Definition \ref{d12}.
In this paper we assume that $\er$ and $\mu$ are scalar functions
(the medium in the cylinder is isotropic), bounded and positive definite,
\begin{equation}
\label{01}
0 < \er_0 \le \er(x) \le \er_1 , \qquad 
0 < \mu_0 \le \mu(x) \le \mu_1 ,
\end{equation} 
and that they depend on the longitudinal variable only,
\begin{equation}
\label{02}
\er (x)= \er(x_3), \qquad \mu (x) = \mu (x_3) .
\end{equation}

In this situation, we show that the Maxwell operator is an orthogonal sum of countable set
of matrix one-dimensional differential operators of first order in $L_2(\R)$ (see Theorem \ref{t31} below).
If the coefficients $\er$, $\mu$ possess the second Sobolev derivatives,
\begin{equation}
\label{03}
\er, \mu \in W_{1,loc}^2 (\R), \qquad
\text{and moreover}, \quad
\sup_{s\in\R} \int_s^{s+1} \left(|\er''(t)| + |\mu''(t)|\right) dt < \infty ,
\end{equation}
then the square ${\cal M}^2$ of the Maxwell operator is unitarily equivalent to the orthogonal sum
of countable set of Schr\" odinger operators 
$-d^2/dy^2 + V_k(y)$ on the real line (Theorem \ref{t47} below).
The Maxwell operator itself is unitarily equivalent to the orthogonal sum
of square roots of Schr\" odinger operators on the real line, taken with signs plus and minus 
(Corollary \ref{c48}).
Note, that the potentials $V_k$ can be explicitely calculated in terms of the coefficients $\er$, $\mu$,
and the eigenvalues $\la_k$, $\ka_l$ of the Laplace operator of the Dirichlet or Neumann boundary problems
in the cross-section $U$ (see \eqref{4e}, \eqref{4m}, \eqref{40} below).

In the precedent work \cite{F20} we have done partial separation of variables,
and we have shown that the square of the Maxwell operator
is unitarily equivalent to the orthogonal sum of four scalar elliptic operators of second order in the cylinder $\Pi$.
Now, we succeeded in complete separating of variables, and we reduce the problem 
to the set of one-dimensional Shr\" odinger operators
that are very well studied.
Note also that in \cite{F20} we used the boundedness of the inverse operator ${\cal M}^{-1}$
which needs the simple connectivity of the cross-section $U$.
In the present paper we do not need such condition.

This reduction to the Shr\" odinger operators allows us to describe the spectrum of the Maxwell operator
for different behaviour of the coefficients.
We consider two cases as natural examples.
In the first case, the coefficient stabilize at infinity,
$$
\er(x_3) \to \er_*, \qquad \mu(x_3) \to \mu_* \qquad \text{as} \quad x_3 \to \pm \infty,
$$
see Theorem \ref{t15} below.
In the second case, the coefficients are periodic along the axis of the cylinder,
see Theorem \ref{t16} below.

\section{Formulation of the results}
\label{S1}
\subsection{Functional spaces}
Let $U \subset \R^2$ be a bounded connected domain, $\dd U \in \lip$, $\Pi = U \times \R$.
We treat the boundary conditions in the definition of the Maxwell operator in the sense of integral identities.

\begin{defi}
\label{d11}
Let $u \in L_2 (\Pi, \C^3)$. 

$\nu$) If $\div u \in L_2 (\Pi)$ then
$$
\left. u_\nu\right|_{\dd\Pi} = 0 \qquad \Longleftrightarrow
\qquad 
\int_\Pi \<u, \n \oo\> dx = - \int_\Pi \div u\,\overline{\oo} dx \quad 
\forall \ \oo \in W_2^1(\Pi) .
$$

$\tau$) If $\rot u\in L_2 (\Pi, \C^3)$ then
$$
\left. u_\tau\right|_{\dd\Pi} = 0 \qquad \Longleftrightarrow
\qquad 
\int_\Pi \<u, \rot z\> dx = \int_\Pi \<\rot u, z\> dx \quad 
\forall \ z \in L_2(\Pi, \C^3) : \rot z \in L_2 (\Pi, \C^3) .
$$
\end{defi}

Here $\<\,.\,,\,.\,\>$ is a standard scalar product in $\C^3$.

We need also several functional spaces.
Introduce the Hilbert space
$$
H(\rot) =  \left\{ u \in L_2(\Pi,\C^3) : \rot u \in L_2(\Pi,\C^3)\right\}, 
$$
endowed with the norm
$$
\|u\|_{H(\rot)}^2 = \int_\Pi \left(|\rot u|^2 + |u|^2\right) dx,
$$
and its subspace
$$
H(\rot, \tau) =  \left\{ u \in H(\rot) : \left. u_\tau\right|_{\dd\Pi} = 0 \right\} .
$$
The following fact is a simple corollary of the definitions.

\begin{lemma}
\label{Htau}
The set $C_0^\infty (\Pi, \C^3)$ is dense in $H(\rot, \tau)$. 
\end{lemma}

\begin{proof}
Consider the operator $\rot_0$ defined on $\dom \rot_0 = C_0^\infty(\Pi, \C^3)$.
Its adjoint operator is the operator $(\rot_0)^* = \rot$ defined on 
$\dom \rot = H(\rot)$.
Therefore, the closure of the operator $\rot_0$ is the operator $\overline{\rot_0} = (\rot)^*$
which acts also as the differential operation $\rot$ on the domain $H(\rot,\tau)$
(see Definition \ref{d11}).
Therefore, the set $C_0^\infty (\Pi, \C^3)$ is dense in $H(\rot, \tau)$ with respect to the graph norm.
\end{proof}

Now, let $\er$, $\mu$ satisfy \eqref{01}.
Introduce the subspaces of divergence-free fucntions 
$$
J(\er) = \{ u \in L_2 (\Pi, \C^3, \er) : \div(\er u) = 0 \},
$$
$$
J(\nu, \mu) = \{ v \in L_2 (\Pi, \C^3, \mu) : \div(\mu v) = 0, 
\left. (\mu v)_\nu\right|_{\dd\Pi} = 0 \}.
$$
Next, we introduce the spaces
$$
\Phi (\tau,\er) = H(\rot,\tau) \cap J(\er), \qquad \Phi (\nu,\mu) = H(\rot) \cap J(\nu,\mu) .
$$

\begin{defi}
\label{d12}
The Maxwell operator is defined in the Hilbert space ${\cal J} : = J(\er) \oplus J(\nu, \mu)$ via the formula
\begin{equation*}
{\cal M} 
\left( \begin{array}{cc} E \\ H \end{array} \right) =
\left( \begin{array}{cc}
i \er^{-1} \rot H \\ -i \mu^{-1} \rot E
\end{array} \right) 
\end{equation*}
on the domain
$$
\dom {\cal M} = \Phi (\tau,\er) \oplus \Phi (\nu,\mu).
$$
\end{defi}

It is easy to see that the Maxwell operator is self-adjoint, ${\cal M} = {\cal M}^*$,
see for example \cite{BS89}.
It has a block structure 
$$
{\cal M} = 
\left( \begin{array}{cc}
0 & R^* \\
R & 0 \end{array} \right).
$$
Here the operators
$$
R =  -i \mu^{-1} \rot, \qquad R^* = i \er^{-1} \rot
$$
defined on  
$$
\dom R =  \Phi (\tau,\er), \qquad \dom R^* = \Phi (\nu,\mu),
$$
are mutually adjoint.
This structure yields that the operators ${\cal M}$ and $-{\cal M}$ are unitarily equivalent,
and the spectrum of the operator ${\cal M}$ is symmetric with respect to zero
(see Lemma \ref{l32} below).

\subsection{Laplace operator in the cross-section}
Denote by $-\D_{\cal D}$ (resp. $-\D_{\cal N}$) the Laplace operator in $U$ with Dirichlet (resp. Neumann) boundary condition.
These are self-adjoint operators corresponding to the quadratic forms
$$
\int_U |\n \ph|^2 dx_1 dx_2, \qquad \ph \in \mathring W_2^1 (U)
$$
and
$$
\int_U |\n \psi|^2 dx_1 dx_2, \qquad \psi \in W_2^1 (U)
$$
respectively.
Cross-section $U$ is a bounded domain with Lipschitz boundary, 
so the spectra of both problems are discrete.
Denote by $\la_k$, $\ph_k$ (resp. $\ka_l$, $\psi_l$) the eigenvalues and the eigenfunctions of the Laplace operator
with Dirichlet (resp. Neumann) boundary condition,
$$
-\D_{\cal D} \ph_k = \la_k \ph_k, \qquad 
-\D_{\cal N} \psi_l = \ka_l \psi_l ,
$$
$$
0 < \la_1 < \la_2 \le \dots, \quad \la_k \to +\infty, 
\qquad
0 = \ka_1 < \ka_2 \le \dots, \quad \ka_l \to +\infty,
$$
$\psi_1 \equiv \const$.
We choose eigenfunctions to be orthonormal,
\begin{equation}
\label{basis}
\|\ph_k\|_{L_2(U)}^2 = 1,\quad \|\n\ph_k\|_{L_2(U)}^2 = \la_k, \qquad
\|\psi_l\|_{L_2(U)}^2 = 1,\quad \|\n\psi_l\|_{L_2(U)}^2 = \ka_l.
\end{equation}
If the boundary $\dd U$ is smooth then the eigenfunctions are classical solutions to the corresponding problems
$$
\begin{cases}
-\D \ph_k = \la_k \ph_k, \\
\left.\ph_k\right|_{\dd U} = 0,
\end{cases},
\qquad
\begin{cases}
-\D \psi_l = \ka_l \psi_l, \\
\left.\frac{\dd\psi_l}{\dd\nu}\right|_{\dd U} = 0.
\end{cases}
$$

\subsection{Results}
\begin{theorem}
\label{t13}
Let $U \subset \R^2$ be a bounded connected domain, $\dd U \in \lip$, $\Pi = U \times \R$.
Let the coefficients  $\er$, $\mu$ be scalar real measurable functions satisfying \eqref{01} and \eqref{02}.
Then the square ${\cal M}^2$ of the Maxwell operator is unitarily equivalent to the orthogonal sum
\begin{equation}
\label{11}
\left(\bigoplus_{k=1}^\infty A^{el}_k\right) \bigoplus \left(\bigoplus_{k=1}^\infty A^{el}_k\right) 
\bigoplus \left(\bigoplus_{l=2}^\infty A^m_l\right) \bigoplus \left(\bigoplus_{l=2}^\infty A^m_l\right)
\bigoplus \left(\bigoplus_{j=1}^{2N-2} A^0\right) ,
\end{equation}
and
\begin{equation}
\label{sp}
\si \left({\cal M}^2\right) = \left(\bigcup_{k=1}^\infty \si \left(A^{el}_k\right)\right)
\bigcup \left(\bigcup_{l=2}^\infty \si \left(A^m_l\right)\right)
\bigcup \si \left(A^0\right).
\end{equation}
Here
$$
A^{el}_k = - \frac1\mu \frac{d}{dz} \left(\frac1\er \frac{d}{dz}\right) + \frac{\la_k}{\er\mu}, 
\qquad \dom A^{el}_k = \left\{p\in W_2^1(\R) : (\er^{-1} p')' \in L_2(\R)\right\}
$$
and 
$$
A^0 = - \frac1\mu \frac{d}{dz} \left(\frac1\er \frac{d}{dz}\right),
\qquad \dom A^0 = \left\{p\in W_2^1(\R) : (\er^{-1} p')' \in L_2(\R)\right\}
$$
are self-adjoint operators in the space $L_2(\R, \mu dz)$;
$$
A^m_l = - \frac1\er \frac{d}{dz} \left(\frac1\mu \frac{d}{dz}\right) + \frac{\ka_l}{\er\mu}, 
\qquad \dom A^m_l = \left\{q\in W_2^1(\R) : (\mu^{-1} q')' \in L_2(\R)\right\}
$$
are self-adjoint operators in the space $L_2(\R, \er dz)$;
$N$ is the number of connected components of the boundary $\dd U$;
if the cross-section $U$ is simply connected, the last summand in \eqref{11} is absent.
\end{theorem}

\begin{rem}
The operator $A^0$ coincides with the operator $A^{el}_k$ after the substitution $\la_k \mapsto 0$.
One can take the operator $A^m_l$ with substitution $\ka_l \mapsto 0$ instead of it.
These operators are unitarily equivalent (see Remark \ref{r48} below).
\end{rem}

\begin{cor}
\label{c14}
If the cross-section $U$ is simply connected then there is a gap centered at zero in the spectrum of the Maxwell operator,
$$
\si({\cal M}) \subset \left(-\infty, -\sqrt{\frac{\ka_2}{\|\er\mu\|_{L_\infty}}}\right] \cup 
\left[\sqrt{\frac{\ka_2}{\|\er\mu\|_{L_\infty}}}, +\infty\right) .
$$
\end{cor}

\begin{theorem}
\label{t15}
Let the conditions of Theorem \ref{t13} be fulfilled.
Assume moreover that there are two constants $\er_*$ and $\mu_*$ such that
$$
\er - \er_* \in W_1^2 (\R), \qquad \mu - \mu_* \in W_1^2(\R) .
$$
Then

1) there are no singular continuous component in the spectrum of the Maxwell oeprator,
$$
\si_{sc} ({\cal M}) = \emptyset .
$$

2) If the cross-section $U$ is multiply connected then the absolute continuous spectrum fills the whole real line,
$$
\si_{ac} ({\cal M}) = \R .
$$
If the cross-section $U$ is simply connected then
$$
\si_{ac} ({\cal M}) =  \left(-\infty, -\sqrt{\frac{\ka_2}{\er_*\mu_*}}\right] \cup 
\left[\sqrt{\frac{\ka_2}{\er_*\mu_*}}, +\infty\right) .
$$

3) If 
$$
\er (z) \mu(z) \le \er_* \mu_* \qquad \forall \ z \in \R,
$$
then there are no eigenvalues in the spectrum, 
$\si_p ({\cal M}) = \emptyset$.
If 
$$
\exists \ z_0 \in \R : \quad \er(z_0) \mu (z_0) > \er_* \mu_*,
$$
then the operator ${\cal M}$ has infinite (countable) set of eigenvalues.
In particular, there are infinitely many eigenvalues inside the continuous spectrum.
\end{theorem}

\begin{theorem}
\label{t16}
Let the conditions of Theorem \ref{t13} be fulfilled.
Assume moreover that $\er, \mu \in W_{1, loc}^2 (\R)$ and that they are periodic,
\begin{equation}
\label{04}
\er (z+a) = \er (z), \qquad \mu (z+a) = \mu (z) .
\end{equation}
Then the spectrum of the Maxwell operator is absolutely continuous,
$$
\si({\cal M}) = \si_{ac}({\cal M}), \qquad  \si_{sc} ({\cal M}) =\emptyset,
\qquad \si_p({\cal M}) = \emptyset ,
$$
and the number of gaps in the spectrum is finite.
\end{theorem}

\subsection{Comments}
{\bf 1.} It is interesting to compare the problem under consideration with the inverse situation.
The Maxwell operator in a cylinder with matrix-valued coefficients that depend
on the cross-sectional variables only was considered in \cite{F18}.
The spectrum of such operator is absolutely continuous.
If the cross-section is multiply connected then the spectrum fills the whole real line;
if the cross-section is simply connected then the spectrum has exactly one gap centered at zero.
Note that the dependence of the spectrum of the Maxwell operator in a cylinder on the topology
of the cross-section is well known in the case of smooth boundary and trivial coefficients
$\er (x) \equiv \mu (x) \equiv \1$ (vacuum in the cylinder),
see \cite[\S 91]{L8}.
In \cite{F18} this statement is generalised to the case of Lipschitz boundary 
and non-trivial coefficients.

{\bf 2.} Under the conditions of Theorem \ref{t15} we do not know 
if the eigenvalues outside the continuous spectrum can occur.

{\bf 3.} The common wisdom assumes that if a periodic problem admits a separation of variables
then the Bethe-Sommerfeld conjecture (the number of gaps in the spectrum is finite)
holds true.
Theorem \ref{t16} supports this point.

\subsection{Plan of the paper}
In \S\S 2 and 3 we separate the variables.
In \S 4 we prove Theorem \ref{t13}.
The key observation here is that the Maxwell operator maps the functions of type
$\left( \begin{array}{cc} a(x_3) \dd_1\ph_k \\
a(x_3) \dd_2 \ph_k \\
c(x_3) \ph_k \\
f(x_3) \dd_2 \ph_k \\
- f(x_3) \dd_1 \ph_k \\
0\end{array} \right)$ into the functions of the same type;
and the same is true for the functions of type 
$\left( \begin{array}{cc} b(x_3) \dd_2 \psi_l \\
- b(x_3) \dd_1 \psi_l \\
0 \\
e(x_3) \dd_1 \psi_l \\
e(x_3) \dd_2 \psi_l \\
g(x_3) \psi_l
\end{array} \right)$.
In \S 5 we prove Corollary \ref{c14}.
In \S 6 we transform the operators 
$ A^{el}_k$, $A^m_l$, $A^0$ introduced in Theorem \ref{t13} 
into the Schr\" odinger type operators.
Here we assume that the coefficients are twice differentiable.
Finally, we prove Theorem \ref{t15} in \S 7, and Theorem \ref{t16} in \S 8.

\section{Separation of variables}
\subsection{Decompositions of the space $L_2 (U, \C^2)$}
If $U$ is multiply connected we denote by $\G_j$ the connected components of its boundary,
$\dd U = \cup_{j=1}^N \G_j$.
Introduce the space of harmonic functions in $U$ that are constant on each component of the boundary,
$$
{\cal L} = \left\{ \oo \in W_2^1 (U) : \D \oo = 0, \left. \oo \right|_{\G_j} = c_j, j = 1, \dots, N\right\} .
$$
Clearly, $\dim {\cal L} = N$.
Put 
$$
{\cal H}^0 = \left\{ \n \oo : \oo \in {\cal L} \right\}.
$$
The identical constant does not make a contribution here, so $\dim {\cal H}^0 = N-1$.
Let us fix an orthonormal basis $\{\n\oo_j\}_{j=1}^{N-1}$ in ${\cal H}^0$.

\begin{lemma}
\label{l20}
If $\ph \in \mathring W_2^1 (U)$, $\eta \in W_2^1 (U)$, then
\begin{equation}
\label{21}
\int_U \left(\dd_1 \ph \dd_2 \eta - \dd_2 \ph \dd_1 \eta\right) dx_1 dx_2 = 0 .
\end{equation}
If $\oo \in {\cal L}$, $\eta \in W_2^1 (U)$, then
\begin{equation}
\label{215}
\int_U \left(\dd_1 \oo \dd_2 \eta - \dd_2 \oo \dd_1 \eta\right) dx_1 dx_2 = 0 .
\end{equation}
\end{lemma}

\begin{proof}
Clearly, \eqref{21} holds for $\eta\in W_2^1(U)$ and $\ph \in C_0^\infty (U)$.
By continuity it is also true for $\ph\in \mathring W_2^1 (U)$.

Fix a function $\oo \in {\cal L}$.
Represent it as a sum $\oo = \tilde \oo + \hat \oo$,
where $\tilde \oo \in C^\infty (\overline U)$ 
and $\tilde \oo$ is constant in a neighbourhood of each component $\G_j$ of the boundary,
$j = 1, \dots, N$;
and $\hat \oo \in \mathring W_2^1 (U)$.
Then 
$$
\int_U \left(\dd_1 \hat \oo \dd_2 \eta - \dd_2 \hat \oo \dd_1 \eta\right) dx_1 dx_2 = 0 
$$
by \eqref{21}.
Furthermore, denote by $\nu$ the unit outward normal vector to the boundary.
Then
$$
\int_U \left(\dd_1 \tilde \oo \dd_2 \eta - \dd_2 \tilde \oo \dd_1 \eta\right) dx_1 dx_2 = 
\int_{\dd U} \left(\dd_1 \tilde \oo\, \nu_2 - \dd_2 \tilde \oo\, \nu_1\right) \eta\, dx_1 dx_2 = 0,
$$
due to the identity $\n \tilde \oo \equiv 0$ in the neighbourhood of $\dd U$.
Now \eqref{215} follows.
\end{proof}

\begin{lemma}
\label{l21}
The set of vector-functions
\begin{equation}
\label{20} 
\left\{ \left( \begin{array}{cc} \dd_1\ph_k \\ \dd_2 \ph_k \end{array} \right) \right\}_{k=1}^\infty, \qquad 
\left\{ \left( \begin{array}{cc} \dd_2\psi_l \\ -\dd_1 \psi_l \end{array} \right) \right\}_{l=2}^\infty, \qquad 
\left\{ \left( \begin{array}{cc} \dd_1\oo_j \\ \dd_2 \oo_j \end{array} \right) \right\}_{j=1}^{N-1}
\end{equation} 
is an orthogonal basis in $L_2(U, \C^2)$.
If the cross-section $U$ is simply connected the third set is absent.
\end{lemma}

\begin{proof}
Orthogonality.
Clearly, $\n\ph_k \perp \n\ph_m$ and $\n\psi_k \perp \n\psi_m$ in $L_2(U, \C^2)$ if $k \neq m$.
Lemma \ref{l20} implies that all functions $\n \ph_k$ and $\n \oo_j$ are orthogonal to all functions
$\left(\begin{array}{cc} \dd_2\psi_l \\ -\dd_1 \psi_l \end{array}\right)$.
Furthermore,
$$
\int_U \left(\dd_1 \ph_k \dd_1 \overline{\oo_j} + \dd_2 \ph_k \dd_2 \overline{\oo_j}\right) dx_1 dx_2
= - \int_U \ph_k \overline{\D\oo_j} dx_1 dx_2 = 0
$$
by definition of the space ${\cal H}^0$.

Completeness.
Assume that a vector-function $f \in L_2(U, \C^2)$ is orthogonal to all functions
$\left( \begin{array}{cc} \dd_1\ph_k \\ \dd_2 \ph_k \end{array} \right)$, $k = 1, 2, \dots$, and
$\left(\begin{array}{cc} \dd_2\psi_l \\ -\dd_1 \psi_l \end{array}\right)$, $l=2, 3, \dots$.
Then 
$$
\int_U \left(f_1 \dd_2 \psi - f_2 \dd_1 \psi\right) dx_1 dx_2 = 0 \qquad \forall \ \psi \in W_2^1 (U),
$$
because $\{\psi_l\}_{l=1}^\infty$ is a basis in $W_2^1(U)$, and $\psi_1 \equiv \const$.
Put
$$
\tilde f(x_1, x_2) = 
\begin{cases}
f(x_1, x_2), & \text{if} \ (x_1, x_2) \in U,\\
0, & \text{if} \ (x_1,x_2) \in \R^2 \setminus U .
\end{cases}
$$
Then
$$
\int_{\R^2} \left(\tilde f_1 \dd_2 \psi - \tilde f_2 \dd_1 \psi\right) dx_1 dx_2 = 0 \qquad \forall \ \psi \in W_2^1 (\R^2),
$$
and therefore, there is a function $\tilde \oo \in W_{2,loc}^1 (\R^2)$ such that $\tilde f = \n \tilde \oo$.
Moreover, the function $\tilde \oo$ is constant on the connected components of $\R^2 \setminus U$.
This means that $f = \n \oo$ where $\oo = \left. \tilde \oo \right|_U$, and
$$
\oo \in  W_2^1 (U), \qquad \left. \oo\right|_{\G_j} = \const \quad \forall \ j = 1, \dots, N .
$$
Finally, the condition $f \perp \n \ph_k$ for all $k \in \N$ implies $\div f = 0$ in $U$.
Thus, $\D \oo = 0$, $\oo \in {\cal L}$ and $f \in {\cal H}^0$.
\end{proof}

By 90 degree rotation we obtain 
\begin{cor}
\label{c22}
The set of vector-functions
\begin{equation}
\label{22} 
\left\{ \left( \begin{array}{cc} \dd_1\psi_l \\ \dd_2 \psi_l \end{array} \right) \right\}_{l=2}^\infty, \qquad 
\left\{ \left( \begin{array}{cc} \dd_2\ph_k \\ -\dd_1 \ph_k \end{array} \right) \right\}_{k=1}^\infty, \qquad 
\left\{ \left( \begin{array}{cc} \dd_2\oo_j \\ -\dd_1 \oo_j \end{array} \right) \right\}_{j=1}^{N-1}
\end{equation} 
is an orthogonal basis in $L_2(U, \C^2)$.
If the cross-section $U$ is simply connected the third set is absent.
\end{cor}

\begin{rem}
Functions $\left( \begin{array}{cc} \dd_1\oo_j \\ \dd_2 \oo_j \end{array} \right)$
solve the electrostatic problem in $U$
$$
\operatorname{curl} f = 0, \qquad \div f = 0, \qquad \left.f_\tau\right|_{\dd U} = 0.
$$
Functions $\left( \begin{array}{cc} \dd_2\oo_j \\ -\dd_1 \oo_j \end{array} \right)$
solve magnetostatic problem in $U$
$$
\operatorname{curl} f = 0, \qquad \div f = 0, \qquad \left.f_\nu\right|_{\dd U} = 0.
$$
Here the boundary conditions $\left.f_{\tau,\nu}\right|_{\dd U} = 0$ understood 
in the sence of integral identities analoguous to Definition \ref{d11}
(see \eqref{A1} below).
\end{rem}

\subsection{Key isomorphisms}
Starting from this point we use the following agreement:
the indices $k$ and $m$ belong to $\N$,
index $l \ge 2$, index $j \in \{1, \dots, N-1\}$.

Introduce the auxiliary Hilbert spaces ${\cal H}_{el}$ and ${\cal H}_m$.
Elements of ${\cal H}_{el}$ are the sets of functions from $L_2(\R)$
$$
\{a_k\}_{k=1}^\infty, \quad \{b_l\}_{l=2}^\infty,
\quad \{c_m\}_{m=1}^\infty, \quad \{d_j\}_{j=1}^{N-1}
$$
with finite norm
$$
\int_\R \left(\sum_{k=1}^\infty \la_k |a_k(z)|^2 + \sum_{l=2}^\infty \ka_l |b_l(z)|^2
+ \sum_{m=1}^\infty |c_m(z)|^2 + \sum_{j=1}^{N-1} |d_j(z)|^2\right) \er(z) \, dz .
$$
Elements of ${\cal H}_m$ are the sets of functions from $L_2(\R)$
$$
\{e_l\}_{l=2}^\infty, \quad \{f_k\}_{k=1}^\infty,
\quad \{g_m\}_{m=1}^\infty, \quad \{h_j\}_{j=1}^{N-1}
$$
with finite norm
$$
\int_\R \left(\sum_{l=2}^\infty \ka_l |e_l(z)|^2 + \sum_{k=1}^\infty \la_k |f_k(z)|^2
+ \sum_{m=1}^\infty |g_m(z)|^2 + \sum_{j=1}^{N-1} |h_j(z)|^2\right) \mu(z) \, dz .
$$

Introduce the operator
$$
I_{el} : L_2 (\Pi, \C^3; \er dx) \to {\cal H}_{el}
$$
acting by the following rule.
Let $E \in L_2 (\Pi, \C^3)$.
For every $x_3 \in \R$ we expand the vector-function
$\left( \begin{array}{cc} E_1(x) \\ E_2(x) \end{array} \right)$
in basis \eqref{20}, and expand the scalar function $E_3(x)$ in basis $\{\ph_k\}_{k=1}^\infty$:
\begin{equation}
\label{23}
\left( \begin{array}{cc} E_1(x) \\ E_2(x) \end{array} \right) =
\sum_{k=1}^\infty a_k (x_3) \left( \begin{array}{cc} \dd_1\ph_k \\ \dd_2 \ph_k \end{array} \right) 
+ \sum_{l=2}^\infty b_l (x_3) \left( \begin{array}{cc} \dd_2\psi_l \\ -\dd_1 \psi_l \end{array} \right) 
+ \sum_{j=1}^{N-1} d_j(x_3) \left( \begin{array}{cc} \dd_1\oo_j \\ \dd_2 \oo_j \end{array} \right),
\end{equation}
\begin{equation}
\label{24}
E_3(x) = \sum_{m=1}^\infty c_m (x_3) \ph_m (x_1, x_2) ,
\end{equation}
and define
$$
I_{el} E = \left\{ \{a_k\}_{k=1}^\infty, \{b_l\}_{l=2}^\infty, \{c_m\}_{m=1}^\infty, \{d_j\}_{j=1}^{N-1} \right\}.
$$
Clearly,
\begin{eqnarray*}
\int_\Pi |E(x)|^2 \er(x) dx 
= \int_\R \left(\sum_{k=1}^\infty \la_k |a_k(x_3)|^2 + \sum_{l=2}^\infty \ka_l |b_l(x_3)|^2 \right.\\
\left.+ \sum_{m=1}^\infty |c_m(x_3)|^2 + \sum_{j=1}^{N-1} |d_j(x_3)|^2\right) \er(x_3) \, dx_3 ,
\end{eqnarray*}
thus, the operator $I_{el}$ is an isometric isomorphism of Hilbert spaces.

In the same way, introduce the operator
$$
I_m : L_2 (\Pi, \C^3; \mu dx) \to {\cal H}_m
$$
Let $H \in L_2 (\Pi, \C^3)$.
For every $x_3 \in \R$ we expand the vector-function
$\left( \begin{array}{cc} H_1(x) \\ H_2(x) \end{array} \right)$
in basis \eqref{22}, and expand the scalar function $H_3(x)$ in basis $\{\psi_m\}_{m=1}^\infty$:
\begin{equation}
\label{25}
\left( \begin{array}{cc} H_1(x) \\ H_2(x) \end{array} \right) =
\sum_{l=2}^\infty e_l (x_3) \left( \begin{array}{cc} \dd_1\psi_l \\ \dd_2 \psi_l \end{array} \right) 
+ \sum_{k=1}^\infty f_k (x_3) \left( \begin{array}{cc} \dd_2\ph_k \\ -\dd_1 \ph_k \end{array} \right) 
+ \sum_{j=1}^{N-1} h_j(x_3) \left( \begin{array}{cc} \dd_2\oo_j \\ -\dd_1 \oo_j \end{array} \right),
\end{equation}
\begin{equation}
\label{26}
H_3(x) = \sum_{m=1}^\infty g_m (x_3) \psi_m (x_1, x_2) ,
\end{equation}
and define
$$
I_m H = \left\{ \{e_l\}_{l=2}^\infty, \{f_k\}_{k=1}^\infty, \{g_m\}_{m=1}^\infty, \{h_j\}_{j=1}^{N-1} \right\}.
$$
We have
\begin{eqnarray*}
\int_\Pi |H(x)|^2 \mu(x) dx 
= \int_\R \left(\sum_{l=2}^\infty \ka_l |e_l(x_3)|^2 + \sum_{k=1}^\infty \la_k |f_k(x_3)|^2 \right.\\
\left.+ \sum_{m=1}^\infty |g_m(x_3)|^2 + \sum_{j=1}^{N-1} |h_j(x_3)|^2\right) \mu(x_3) \, dx_3 ,
\end{eqnarray*}
so, the operator $I_m$ is an isometric isomorphism of Hilbert spaces.

Next, we describe the action of the operators $I_{el}$, $I_m$ on the different spaces from Section 1.1.

\subsection{Spaces of divergence-free functions}
\begin{lemma}
\label{l23}
The following identities take place
\begin{equation}
\label{27}
I_{el} J(\er) = 
\left\{ \{a_k, b_l, c_m, d_j\} \in {\cal H}_{el} :
\er c_m \in W_2^1(\R), (\er c_m)' = \la_m \er a_m \ \ \forall m \in \N \right\} ;
\end{equation}
\begin{equation}
\label{28}
I_m J(\nu, \mu) = 
\left\{ \{e_l, f_k, g_m, h_j\} \in {\cal H}_m : g_1 \equiv 0,
\mu g_l \in W_2^1(\R), (\mu g_l)' = \ka_l \mu e_l \ \ \forall l \ge 2 \right\} .
\end{equation}
\end{lemma}

\begin{proof}
Let $E \in L_2(\Pi, \C^3)$.
Expand it in the series \eqref{23}, \eqref{24}.
Fix $k \in \N$ and $\eta \in C_0^\infty (\R)$.
Using Lemma \ref{l21} and the identities \eqref{basis},
we obtain
\begin{equation}
\label{285}
\int_\Pi \left\< \er(x_3) E(x), \n \left(\ph_k (x_1, x_2) \eta (x_3)\right) \right\> dx
= \int_\R \left(\la_k \er(x_3) a_k(x_3) \overline{\eta(x_3)} + \er(x_3) c_k(x_3) \overline{\eta'(x_3)}\right) dx_3.
\end{equation}
The set of linear combinations of functions of type $\ph_k (x_1, x_2) \eta (x_3)$ is dense in $\mathring W_2^1 (\Pi)$.
Therefore, the equality $\div (\er E) = 0$ is equivalent to the condition that the right hand sides of \eqref{285} 
vanish for all $\eta \in C_0^\infty (\R)$ and for all $k \in \N$.
This means 
$$
\exists \ (\er c_k)' = \la _k \er a_k \qquad \forall \ k \in \N .
$$
This implies \eqref{27}.

Now, let $H \in L_2(\Pi, \C^3)$.
Expand it in the series \eqref{25}, \eqref{26}.
Fix $m \in \N$ and $\eta \in C_0^\infty (\R)$.
Using Corollary \ref{c22} and the identities \eqref{basis},
we obtain
\begin{eqnarray}
\nonumber
\int_\Pi \left\< \mu(x_3) H(x), \n \left(\psi_m (x_1, x_2) \eta (x_3)\right) \right\> dx\\
= \begin{cases}
\int_\R \mu(x_3) g_1 (x_3) \overline{\eta'(x_3)} dx_3, & \text{if} \ m =1,\\
\int_\R \left(\ka_m \mu(x_3) e_m (x_3) \eta (x_3) + \mu(x_3) g_m (x_3) \overline{\eta'(x_3)}\right) dx_3,
& \text{if} \ m \ge 2 .
\end{cases}
\label{287}
\end{eqnarray}
The set of linear combinations of functions of type $\psi_m (x_1, x_2) \eta (x_3)$ is dense in $W_2^1 (\Pi)$.
Therefore, the equalities $\div (\mu H) = 0$ and $\left(\mu H)_\nu\right|_{\dd \Pi} = 0$
are equivalent to the condition that the right hand sides of \eqref{287} 
vanish for all $\eta \in C_0^\infty (\R)$ and for all $m \in \N$.
Note, that for a function $g_1 \in L_2(\R)$ the conditions
$(\mu g_1)' \equiv 0$ and $g_1 \equiv 0$ are equivalent.
This implies \eqref{28}.
\end{proof} 

\section{Spaces $H(\rot)$, $H(\rot, \tau)$}
\subsection{"Basis" in $H(\rot,\tau)$}

\begin{lemma}
\label{l241}
Let 
$$
u \in L_2(U, \C^2), \qquad \div u \in L_2(U), 
\qquad \left. u_\nu\right|_{\dd U} = 0 .
$$
We understand the last equality in the sense analoguous to Definition \ref{d11}:
\begin{equation}
\label{A1}
\int_U \< u, \n \ze \> dx_1 dx_2 
= - \int_U \div u\, \overline \ze\, dx_1 dx_2 \qquad \forall \ \ze \in W_2^1(U) .
\end{equation}
Then there is a sequence of functions $v^{(n)} \in C_0^\infty (U, \C^2)$ such that
$$
v^{(n)} \to u \ \ \text{in} \ \ L_2(U, \C^2) \qquad \text{and} \qquad 
\div v^{(n)} \to \div u\ \  \text{in}\ \  L_2(U) .
$$
\end{lemma}

\begin{proof}
Let us consider the operator 
$\div_0 : L_2(U, \C^2) \to L_2(U)$ defined on 
$$
\dom \div_0 = C_0^\infty (U, \C^2)
$$ and acting as divergence,
$\div_0 v = \div v$.
This operator is densely defined and not closed.
Its adjoint operator is the operator nabla
$$
(\div_0)^* =- \n : L_2(U) \to L_2(U, \C^2)
$$
defined on $\dom \n = W_2^1(U)$.
Therefore, the closure of the operator $\div_0$ is the operator adjoint to nabla,
$\overline{\div_0} = -\n^* = \div$.
The last operator is described just by the formula \eqref{A1}. 
So, any function $u \in \dom \n^* = \dom \overline{\div_0}$ can be approximated by the functions from 
$C_0^\infty (U, \C^2)$ in the graph norm of the operator $\div$.
\end{proof}

\begin{lemma}
\label{l242}
Let $\eta \in C_0^\infty (\R)$.
Then the vector-functions
$$
\left( \begin{array}{cc} \dd_1 \ph_k (x_1, x_2) \eta (x_3) \\ 
\dd_2 \ph_k (x_1, x_2) \eta (x_3) \\ 
0 \end{array} \right) \quad
\left( \begin{array}{cc} \dd_2 \psi_l (x_1, x_2) \eta (x_3) \\ 
-\dd_1 \psi_l (x_1, x_2) \eta (x_3) \\ 
0 \end{array} \right), 
$$
$$
\left( \begin{array}{cc} \dd_1 \oo_j (x_1, x_2) \eta (x_3) \\ 
\dd_2 \oo_j (x_1, x_2) \eta (x_3) \\ 
0 \end{array} \right), \quad
\left( \begin{array}{cc} 0 \\ 0 \\ 
\ph_m (x_1, x_2) \eta (x_3) \end{array} \right)
$$
belong to the space $H(\rot, \tau)$, and
\begin{equation}
\label{A2}
\rot \left( \begin{array}{cc} \dd_1 \ph_k (x_1, x_2) \eta (x_3) \\ 
\dd_2 \ph_k (x_1, x_2) \eta (x_3) \\ 
0 \end{array} \right) =
\left( \begin{array}{cc} -\dd_2 \ph_k (x_1, x_2) \eta' (x_3) \\ 
\dd_1 \ph_k (x_1, x_2) \eta' (x_3) \\ 
0 \end{array} \right),
\end{equation}
\begin{equation}
\label{A3}
\rot \left( \begin{array}{cc} \dd_2 \psi_l (x_1, x_2) \eta (x_3) \\ 
-\dd_1 \psi_l (x_1, x_2) \eta (x_3) \\ 
0 \end{array} \right) =
\left( \begin{array}{cc} \dd_1 \psi_l (x_1, x_2) \eta' (x_3) \\ 
\dd_2 \psi_l (x_1, x_2) \eta' (x_3) \\ 
\ka_l \psi_l(x_1,x_2) \eta(x_3) \end{array} \right),
\end{equation}
\begin{equation}
\label{A4}
\rot \left( \begin{array}{cc} \dd_1 \oo_j (x_1, x_2) \eta (x_3) \\ 
\dd_2 \oo_j (x_1, x_2) \eta (x_3) \\ 
0 \end{array} \right) =
\left( \begin{array}{cc} - \dd_2 \oo_j (x_1, x_2) \eta' (x_3) \\ 
\dd_1 \oo_j (x_1, x_2) \eta' (x_3) \\ 
0 \end{array} \right),
\end{equation}
\begin{equation}
\label{A5}
\rot \left( \begin{array}{cc} 0 \\ 0 \\ 
\ph_m (x_1, x_2) \eta (x_3) \end{array} \right) =
\left( \begin{array}{cc} \dd_2 \ph_m (x_1, x_2) \eta (x_3) \\ 
-\dd_1 \ph_m (x_1, x_2) \eta (x_3) \\ 
0 \end{array} \right).
\end{equation}
\end{lemma}

\begin{proof}
In order to prove that a function belong to the space $H(\rot,\tau)$ 
it is sufficient to show that it can be approximated in the norm of $H(\rot)$
by the functions from $C_0^\infty (\Pi, \C^3)$.

1. By definition, $\ph_k \in \mathring W_2^1 (U)$.
Thus, there is a sequence $\ph^{(n)} \in C_0^\infty(U)$ such that 
$\ph^{(n)} \to \ph_k$ in $W_2^1(U)$.
By direct calculation, we have
$$
\rot \left( \begin{array}{cc} \dd_1 \ph^{(n)} (x_1, x_2) \eta (x_3) \\ 
\dd_2 \ph^{(n)} (x_1, x_2) \eta (x_3) \\ 
0 \end{array} \right) =
\left( \begin{array}{cc} -\dd_2 \ph^{(n)} (x_1, x_2) \eta' (x_3) \\ 
\dd_1 \ph^{(n)} (x_1, x_2) \eta' (x_3) \\ 
0 \end{array} \right) .
$$
Therefore, the sequence 
$\left\{  \left( \begin{array}{cc} \dd_1 \ph^{(n)} (x_1, x_2) \eta (x_3) \\ 
\dd_2 \ph^{(n)} (x_1, x_2) \eta (x_3) \\ 
0 \end{array} \right) \right\}$
is a Cauchy sequence in the space $H(\rot)$.
This implies that
$$
 \left( \begin{array}{cc} \dd_1 \ph^{(n)} (x_1, x_2) \eta (x_3) \\ 
\dd_2 \ph^{(n)} (x_1, x_2) \eta (x_3) \\ 
0 \end{array} \right)
\to 
\left( \begin{array}{cc} \dd_1 \ph_k (x_1, x_2) \eta (x_3) \\ 
\dd_2 \ph_k (x_1, x_2) \eta (x_3) \\ 
0 \end{array} \right)
\qquad \text{in} \quad H(\rot),
$$
and \eqref{A2} is fulfilled.

2. Let $l\ge 2$.
By virtue of Lemma \ref{l241}
there is a sequence $v^{(n)} \in C_0^\infty (U, \C^2)$ such that
$$
v^{(n)} \to \n \psi_l \ \ \text{in} \ \ L_2(U, \C^2) \qquad \text{and} \qquad 
\div v^{(n)} \to \D \psi_l\ \  \text{in}\ \  L_2(U) .
$$
We have
$$
\rot \left( \begin{array}{cc} v_2^{(n)} (x_1, x_2) \eta (x_3) \\ 
-v_1^{(n)} (x_1, x_2) \eta (x_3) \\ 
0 \end{array} \right) =
\left( \begin{array}{cc} v_1^{(n)} (x_1, x_2) \eta' (x_3) \\ 
v_2^{(n)} (x_1, x_2) \eta' (x_3) \\ 
- \div v^{(n)} (x_1, x_2) \eta(x_3) \end{array} \right).
$$
Therefore, the sequence 
$ \left( \begin{array}{cc} v_2^{(n)} (x_1, x_2) \eta (x_3) \\ 
-v_1^{(n)} (x_1, x_2) \eta (x_3) \\ 
0 \end{array} \right)$
is a Cauchy sequence in $H(\rot)$.
This implies that it converges to 
$\left( \begin{array}{cc} \dd_2 \psi_l (x_1, x_2) \eta (x_3) \\ 
-\dd_1 \psi_l (x_1, x_2) \eta (x_3) \\ 
0 \end{array} \right)$ in $H(\rot)$,
and \eqref{A3} is fulfilled.

3. Represent a function $\oo_j$ as a sum $\oo_j = \tilde \oo + \hat \oo$,
where $\tilde \oo \in C^\infty (\overline U)$ 
and $\tilde \oo$ is constant in a neighbourhood of each component $\G_i$ of the boundary,
$i = 1, \dots, N$;
and $\hat \oo \in \mathring W_2^1 (U)$.
Vector-function
$\left( \begin{array}{cc} \dd_1 \tilde \oo (x_1, x_2) \eta (x_3) \\ 
\dd_2 \tilde \oo (x_1, x_2) \eta (x_3) \\ 
0 \end{array} \right) 
\in C_0^\infty (\Pi, \C^3)$
by construction.
Vector-function
$\left( \begin{array}{cc} \dd_1 \hat \oo (x_1, x_2) \eta (x_3) \\ 
\dd_2 \hat \oo (x_1, x_2) \eta (x_3) \\ 
0 \end{array} \right)$
can be approximated in the same way as we did at the first step.
Therefore, \eqref{A4} is fulfilled.

4. Again, let  $\ph^{(n)} \in C_0^\infty(U)$,
$\ph^{(n)} \to \ph_m$ in $W_2^1(U)$.
We have
$$
\rot \left( \begin{array}{cc} 0 \\ 0 \\ 
\ph^{(n)} (x_1, x_2) \eta (x_3) \end{array} \right) =
\left( \begin{array}{cc} \dd_2 \ph^{(n)} (x_1, x_2) \eta (x_3) \\ 
-\dd_1 \ph^{(n)} (x_1, x_2) \eta (x_3) \\ 
0 \end{array} \right).
$$
The sequence 
$\left( \begin{array}{cc} 0 \\ 0 \\ 
\ph^{(n)} (x_1, x_2) \eta (x_3) \end{array} \right)$
is a Cauchy sequence in $H(\rot)$.
Therefore, it converges to $\left( \begin{array}{cc} 0 \\ 0 \\ 
\ph_m (x_1, x_2) \eta (x_3) \end{array} \right)$
in $H(\rot)$, and \eqref{A5} is fulfilled.
\end{proof}

\begin{cor}
\label{c243}
The claim of the precedent Lemma remains valid if we change the condition
$\eta \in C_0^\infty (\R)$ by the condition $\eta \in W_2^1(\R)$ in the first three cases,
and by $\eta \in L_2(\R)$ in the last case.
\end{cor}

\begin{proof}
The set $C_0^\infty (\R)$ is dense in $W_2^1(\R)$ and in $L_2(\R)$.
\end{proof}

\subsection{Space $H(\rot, \tau)$}

\begin{lemma}
\label{l244}
Let a function $E \in C_0^\infty (\Pi, \C^3)$ be expanded in the series \eqref{23}, \eqref{24}.
Then $a_k, b_l, c_m, d_j \in C_0^\infty (\R)$,
\begin{eqnarray}
\label{A6}
\rot E = 
\sum_{k=1}^\infty (c_k - a_k') \left( \begin{array}{cc} \dd_2 \ph_k  \\ 
- \dd_1 \ph_k \\ 
0 \end{array} \right)
+ \sum_{l=2}^\infty b_l' \left( \begin{array}{cc} \dd_1 \psi_l \\ 
\dd_2 \psi_l \\ 
0 \end{array} \right) \\
+ \sum_{l=2}^\infty \ka_l b_l \left( \begin{array}{cc} 0 \\ 0 \\ 
\psi_l \end{array} \right)
- \sum_{j=1}^{N-1} d_j' \left( \begin{array}{cc}  \dd_2 \oo_j \\ 
- \dd_1 \oo_j\\ 
0 \end{array} \right),
\nonumber
\end{eqnarray}
and
\begin{equation}
\label{A7}
\int_\R \left(\sum_{k=1}^\infty \la_k |c_k-a_k'|^2 + \sum_{l=2}^\infty \left(\ka_l |b_l'|^2 + \ka_l^2 |b_l|^2\right)\right) dx_3 < \infty .
\end{equation}
\end{lemma}

\begin{proof}
The functions $a_k, b_l, c_m, d_j$ are the Fourier coefficients of $E\in C_0^\infty(\Pi, \C^3)$, therefore,
$a_k, b_l, c_m, d_j \in C_0^\infty (\R)$.

Expand the field $\rot E$ in the series \eqref{25}, \eqref{26},
denoting the corresponding coefficients by $e_l^{(\rot E)}$, $f_k^{(\rot E)}$, $g_m^{(\rot E)}$, $h_j^{(\rot E)}$.

Fix $l \ge 2$.
Then
\begin{eqnarray*}
\ka_l \,e_l^{(\rot E)} (x_3) = 
\int_U \left\< \rot E, \left( \begin{array}{cc} \dd_1 \psi_l \\ 
\dd_2 \psi_l \\ 0 \end{array} \right) \right\> dx_1 dx_2 \\
= \int_U \left((\dd_2 E_3 - \dd_3 E_2) \dd_1 \overline{\psi_l} 
+ (\dd_3 E_1 - \dd_1 E_3) \dd_2 \overline{\psi_l}\right) dx_1 dx_2 .
\end{eqnarray*}
Next,
$$
\int_U \left(\dd_2 E_3 \dd_1 \overline{\psi_l} - \dd_1 E_3 \dd_2 \overline{\psi_l}\right) dx_1 dx_2 = 0,
$$
so,
\begin{eqnarray*}
\ka_l \,e_l^{(\rot E)} (x_3) = 
 \int_U \left( - \dd_3 E_2 \dd_1 \overline{\psi_l} + \dd_3 E_1 \dd_2 \overline{\psi_l}\right) dx_1 dx_2 \\
= \frac{d}{dx_3}  \int_U \left(E_1 \dd_2 \overline{\psi_l} - E_2 \dd_1 \overline{\psi_l}\right) dx_1 dx_2
= \ka_l \,b_l' (x_3).
\end{eqnarray*}
Thus, 
$e_l^{(\rot E)} (x_3) = b_l' (x_3)$.

Let $m\in \N$.
Then
\begin{eqnarray*}
g_m^{(\rot E)} (x_3) =
\int_U (\rot E)_3 (x) \overline{\psi_m(x_1,x_2)}\, dx_1 dx_2 \\
= \int_U (\dd_1 E_2 - \dd_2 E_1) \overline{\psi_m}\, dx_1 dx_2 
= \int_U \left(E_1 \dd_2 \overline{\psi_m} - E_2 \dd_1 \overline{\psi_m}\right) dx_1 dx_2 
= \ka_l b_l (x_3) .
\end{eqnarray*}

Let $k\in \N$.
We have
\begin{eqnarray*}
\la_k \,f_k^{(\rot E)} (x_3) = 
\int_U \left\< \rot E, \left( \begin{array}{cc} \dd_2 \ph_k \\ 
- \dd_1 \ph_k \\ 0 \end{array} \right) \right\> dx \\
= \int_U \left((\dd_2 E_3 - \dd_3 E_2) \dd_2 \overline{\ph_k} 
- (\dd_3 E_1 - \dd_1 E_3) \dd_1 \overline{\ph_k}\right) dx_1 dx_2 ;
\end{eqnarray*}
$$
\int_U \left(\dd_2 E_3 \dd_2 \overline{\ph_k} + \dd_1 E_3 \dd_1 \overline{\ph_k}\right) dx_1 dx_2 = 
- \int_U E_3\, \overline{\D \ph_k}\, dx_1 dx_2
= \la_k \int_U E_3 \,\overline{\ph_k}\, dx_1 dx_2 = \la_k c_k (x_3);
$$
$$
\int_U \left(- \dd_3 E_2 \dd_2 \overline{\ph_k} 
- \dd_3 E_1 \dd_1 \overline{\ph_k}\right) dx_1 dx_2 
= - \frac{d}{dx_3} \int_U \left(E_1 \dd_1 \overline{\ph_k} + E_2 \dd_2 \overline{\ph_k}\right) dx_1 dx_2 
= - \la_k a_k' (x_3) .
$$
Thus,
$$
f_k^{(\rot E)} (x_3) = - a_k' (x_3) + c_k (x_3).
$$

Finally, let $j \in \{1, \dots, N-1\}$.
Then as above
\begin{eqnarray*}
h_j^{(\rot E)} (x_3) = 
\int_U \left\< \rot E, \left( \begin{array}{cc} \dd_2 \oo_j \\ 
- \dd_1 \oo_j \\ 0 \end{array} \right) \right\> dx \\
= - \int_U E_3\, \overline{\D \oo_j}\, dx_1 dx_2
- \frac{d}{dx_3} \int_U \left(E_1 \dd_1 \overline{\oo_j} + E_2 \dd_2 \overline{\oo_j}\right) dx_1 dx_2 
= - d_j'(x_3)
\end{eqnarray*}
because $\D \oo_j \equiv 0$.

We have established \eqref{A6}.
It yields the convergence \eqref{A7}.
\end{proof}

\begin{theorem}
\label{t25}
The identity
\begin{eqnarray}
\nonumber
I_{el} H(\rot,\tau) 
= \left\{ \{a_k, b_l, c_m, d_j\} \in {\cal H}_{el} :
a_k, b_l, d_j \in W_2^1(\R), \right.\\
\left. \int_\R \left(\sum_{k=1}^\infty \la_k |c_k-a_k'|^2 + \sum_{l=2}^\infty \left(\ka_l |b_l'|^2 + \ka_l^2 |b_l|^2\right)\right) dx_3 < \infty\right\} 
\label{210}
\end{eqnarray}
holds.
If the field $E \in H(\rot, \tau)$ is expanded in the series \eqref{23}, \eqref{24},
then \eqref{A6} is fulfilled.
\end{theorem}

\begin{proof}
Let $E \in H(\rot, \tau)$.
By virtue of Lemma \ref{Htau}
there is a sequence $E^{(n)} \in C_0^\infty (\Pi, \C^3)$ such that $E^{(n)} \to E$ in $H(\rot,\tau)$.
Expand the functions $E$ and $E^{(n)}$ in the series \eqref{23}, \eqref{24}.
Denote the corresponding coefficients by $a_k, b_l, c_m, d_j$ and $a_k^{(n)}, b_l^{(n)}, c_m^{(n)}, d_j^{(n)}$.
Clearly, $a_k^{(n)} \to a_k$, $b_l^{(n)} \to b_l$, $c_m^{(n)} \to c_m$ and $d_j^{(n)} \to d_j$ in $L_2(\R)$ as $n \to \infty$.
Moreover,
\begin{eqnarray*}
\int_\R \left(\sum_{k=1}^\infty \la_k \left|c_k^{(n)} - (a_k^{(n)})' - \left(c_k^{(\tilde n)} - (a_k^{(\tilde n)})'\right)\right|^2 \right. \\
\left. + \sum_{l=2}^\infty \left(\ka_l \left|(b_l^{(n)})' - (b_l^{(\tilde n)})'\right|^2 + \ka_l^2 \left|b_l^{(n)} - b_l^{(\tilde n)}\right|^2\right)
+ \sum_{j=1}^{N-1} \left|(d_j^{(n)})' - (d_j^{(\tilde n)})'\right|^2\right) dx_3 \\
= \left\| \rot E^{(n)} - \rot E^{(\tilde n)}\right\|_{L_2(\Pi)}^2 
\mathop{\longrightarrow}\limits_{n, \tilde n \to \infty} 0 .
\end{eqnarray*}
So, the sequences $\{a_k^{(n)}\}$, $\{b_l^{(n)}\}$, $\{d_j^{(n)}\}$ are the Cauchy sequences 
in the space $W_2^1(\R)$.
Therefore, they converge in $W_2^1(\R)$ to $a_k$, $b_l$, and $d_j$ respectively.
As we have $\rot E^{(n)} \to \rot E$ in $L_2(\Pi, \C^3)$,
this implies \eqref{A6} and \eqref{A7} for the function $E$.
Thus, we proved the inclusion $\subset$ in \eqref{210}.
 
Now, let  $\{a_k, b_l, c_m, d_j\} \in {\cal H}_{el}$, $a_k, b_l, d_j \in W_2^1(\R)$ and \eqref{A7} be satisfied.
Define the function $E$ by the formulas \eqref{23}, \eqref{24} with such coefficients.
Any partial sum of this series belong to $H(\rot, \tau)$ and satisfies \eqref{A6} due to Corollary \ref{c243}.
Moreover, the convergence \eqref{A7} yields the convergence of such partial sums in $H(\rot)$.
Therefore, $E \in H(\rot,\tau)$, and we have proved the inclusion $\supset$ in \eqref{210}.
\end{proof}

\subsection{Space $H(\rot)$}

\begin{lemma}
\label{l246}
Let a function $H \in H(\rot)$ be expanded in the series \eqref{25}, \eqref{26}.
Then $e_l, f_k, h_j \in W_2^1 (\R)$,
\begin{eqnarray}
\label{A8}
\rot H = 
\sum_{k=1}^\infty f_k' \left( \begin{array}{cc} \dd_1 \ph_k  \\ 
\dd_2 \ph_k \\ 
0 \end{array} \right)
+ \sum_{l=2}^\infty (g_l-e_l')\left( \begin{array}{cc} \dd_2 \psi_l \\ 
- \dd_1 \psi_l \\ 
0 \end{array} \right) \\
+ \sum_{k=1}^\infty \la_k f_k \left( \begin{array}{cc} 0 \\ 0 \\ 
\ph_k \end{array} \right)
+ \sum_{j=1}^{N-1} h_j' \left( \begin{array}{cc}  \dd_1 \oo_j \\ 
\dd_2 \oo_j\\ 
0 \end{array} \right),
\nonumber
\end{eqnarray}
and
\begin{equation}
\label{A9}
\int_\R \left(\sum_{k=1}^\infty \left(\la_k |f_k'|^2 + \la_k^2 |f_k|^2\right) 
+ \sum_{l=2}^\infty \ka_l |g_l - e_l'|^2 \right) dx_3 < \infty .
\end{equation}
\end{lemma}

\begin{proof}
Expand the field $\rot H$ in the series \eqref{23}, \eqref{24},
denoting the corresponding coefficients by $a_k^{(\rot H)}$, $b_l^{(\rot H)}$, $c_m^{(\rot H)}$, $d_j^{(\rot H)}$.

Let $k \in \N$, $\eta \in C_0^\infty (\R)$.
Then 
$\left( \begin{array}{cc} \dd_1 \ph_k \eta \\ 
\dd_2 \ph_k \eta \\ 0 \end{array} \right) \in H(\rot, \tau)$
due to Lemma \ref{l242}, and the equality \eqref{A2} is valid.
Therefore,
\begin{eqnarray*}
\la_k \int_\R a_k^{(\rot H)} (x_3) \overline{\eta(x_3)} dx_3
= \int_\Pi \left\< \rot H, \left( \begin{array}{cc} \dd_1 \ph_k \eta \\ 
\dd_2 \ph_k \eta \\ 0 \end{array} \right) \right\> dx \\
=  \int_\Pi \left\< H, \rot \left( \begin{array}{cc} \dd_1 \ph_k \eta \\ 
\dd_2 \ph_k \eta \\ 0 \end{array} \right) \right\> dx
=  \int_\Pi \left\< H, \left( \begin{array}{cc} - \dd_2 \ph_k \eta' \\ 
\dd_1 \ph_k \eta' \\ 0 \end{array} \right) \right\> dx
= - \la_k \int_\R f_k (x_3) \overline{\eta'(x_3)} dx_3 .
\end{eqnarray*}
So, $f_k \in W_2^1(\R)$ and $f_k' = a_k^{(\rot H)}$.

Next,
\begin{eqnarray*}
\int_\R c_k^{(\rot H)} (x_3) \overline{\eta(x_3)} dx_3
= \int_\Pi \left\< \rot H, \left( \begin{array}{cc} 0 \\ 0\\ 
\ph_k \eta \end{array} \right) \right\> dx \\
=  \int_\Pi \left\< H, \rot \left( \begin{array}{cc} 0 \\ 0 \\ 
\ph_k \eta \end{array} \right) \right\> dx
=  \int_\Pi \left\< H, \left( \begin{array}{cc} \dd_2 \ph_k \eta \\ 
- \dd_1 \ph_k \eta \\ 0 \end{array} \right) \right\> dx
= \la_k \int_\R f_k (x_3) \overline{\eta(x_3)} dx_3 .
\end{eqnarray*}
Therefore, $c_k^{(\rot H)} = \la_k f_k$.

Now, let $l\ge 2$.
Using Lemma \ref{l242} again we obtain
\begin{eqnarray*}
\ka_l \int_\R b_l^{(\rot H)} (x_3) \overline{\eta(x_3)} dx_3
= \int_\Pi \left\< \rot H, \left( \begin{array}{cc} \dd_2 \psi_l \eta \\ 
- \dd_1 \psi_l \eta \\ 0 \end{array} \right) \right\> dx 
=  \int_\Pi \left\< H, \rot \left( \begin{array}{cc} \dd_2 \psi_l \eta \\ 
- \dd_1 \psi_l \eta \\ 0 \end{array} \right) \right\> dx \\
=  \int_\Pi \left\< H, \left( \begin{array}{cc}  \dd_1 \psi_l \eta' \\ 
\dd_2 \psi_l \eta' \\ \ka_l \psi_l \eta \end{array} \right) \right\> dx
= \ka_l \int_\R \left(e_l (x_3) \overline{\eta'(x_3)} + g_l (x_3) \overline{\eta(x_3)} \right) dx_3 .
\end{eqnarray*}
So, $e_l \in W_2^1(\R)$ and $b_l^{(\rot H)} = g_l - e_l'$.

Finally, for $j \in \{1, \dots, N-1\}$ we have
\begin{eqnarray*}
\int_\R d_j^{(\rot H)} (x_3) \overline{\eta(x_3)} dx_3
= \int_\Pi \left\< \rot H, \left( \begin{array}{cc} \dd_1 \oo_j \eta \\ 
\dd_2 \oo_j \eta \\ 0 \end{array} \right) \right\> dx \\
=  \int_\Pi \left\< H, \rot \left( \begin{array}{cc} \dd_1 \oo_j \eta \\ 
\dd_2 \oo_j \eta \\ 0 \end{array} \right) \right\> dx
=  \int_\Pi \left\< H, \left( \begin{array}{cc} - \dd_2 \oo_j \eta' \\ 
\dd_1 \oo_j \eta' \\ 0 \end{array} \right) \right\> dx
= - \int_\R h_j (x_3) \overline{\eta'(x_3)} dx_3 .
\end{eqnarray*}
Therefore, $h_j \in W_2^1(\R)$ and $h_j' = d_j^{(\rot H)}$.

Thus, we showed \eqref{A8}.
It implies \eqref{A9}.
\end{proof}

\begin{theorem}
\label{t24}
The identity
\begin{eqnarray}
\nonumber
I_m H(\rot) 
= \left\{ \{e_l, f_k, g_m, h_j\} \in {\cal H}_m :
e_l, f_k, h_j \in W_2^1(\R), \right. \\
\left. \int_\R \left(\sum_{k=1}^\infty \left(\la_k |f_k'|^2 + \la_k^2 |f_k|^2\right) 
+ \sum_{l=2}^\infty \ka_l |g_l - e_l'|^2 \right) dx_3 < \infty\right\}
\label{29}
\end{eqnarray}
holds.
\end{theorem}

\begin{proof}
The inclusion $\subset$ is due to Lemma \ref{l246}.

Let $ \{e_l, f_k, g_m, h_j\} \in {\cal H}_m$, $e_l, f_k, h_j \in W_2^1(\R)$ and 
\eqref{A9} be satisfied.
Define the function $H$ by the formulas \eqref{25}, \eqref{26},
and denote by $G$ the right hand side of \eqref{A8}.
We have to prove that $G = \rot H$.

Let the function $E \in C_0^\infty (\Pi, \C^3)$ be expanded in the series \eqref{23}, \eqref{24}.
By virtue of Lemma \ref{l244}
\begin{eqnarray*}
\int_\Pi \left\< H, \rot E \right\> dx
= \int_\R \left( \sum_{k=1}^\infty \la_k f_k (\overline{c_k-a_k'}) 
+ \sum_{l=2}^\infty \ka_l \left( e_l \overline{b_l'} + g_l \overline{b_l} \right) 
- \sum_{j=1}^{N-1} h_j \overline{d_j'} \right) dx_3 \\
= \int_\R \left( \sum_{k=1}^\infty \la_k \left(f_k' \overline{a_k} + f_k \overline{c_k}\right)
+ \sum_{l=2}^\infty \ka_l (g_l - e_l') \overline{b_l}
+ \sum_{j=1}^{N-1} h_j' \overline{d_j} \right) dx_3 
= \int_\Pi \left\<G, E\right\> dx .
\end{eqnarray*}
The second equality here is clear for each summand,
and the series converge due to \eqref{29} and \eqref{A7}.
So,
$$
\int_\Pi \left\< H, \rot E \right\> dx
= \int_\Pi \left\<G, E\right\> dx  \qquad \forall \ \ E \in C_0^\infty (\Pi, \C^3)
$$
which means $H \in H(\rot)$ and $\rot H = G$.
We have proved the inclusion $\supset$ in \eqref{29}.
\end{proof} 

\begin{cor}
\label{c26}
The identities
\begin{eqnarray*}
\nonumber
I_{el} \Phi(\tau, \er) 
= \left\{ \{a_k, b_l, c_m, d_j\} \in {\cal H}_{el} :
a_k, b_l, \er c_m, d_j \in W_2^1(\R), (\er c_m)' = \la_m \er a_m \ \ \forall m \in \N \right.\\
\left. \int_\R \left(\sum_{k=1}^\infty \la_k |c_k-a_k'|^2 + \sum_{l=2}^\infty \left(\ka_l |b_l'|^2 + \ka_l^2 |b_l|^2\right)\right) dx_3 < \infty\right\} 
\end{eqnarray*}
and
\begin{eqnarray*}
I_m \Phi(\nu, \mu) 
= \left\{ \{e_l, f_k, g_m, h_j\} \in {\cal H}_m :
e_l, f_k, \mu g_m, h_j \in W_2^1(\R), g_1 \equiv 0, (\mu g_l)' = \ka_l \mu g_l \ \ \forall l \ge 2 \right. \\
\left. \int_\R \left(\sum_{k=1}^\infty \left(\la_k |f_k'|^2 + \la_k^2 |f_k|^2\right) 
+ \sum_{l=2}^\infty \ka_l |g_l - e_l'|^2 \right) dx_3 < \infty\right\}
\end{eqnarray*}
hold.
\end{cor}

\section{Decomposition of the Maxwell operator}
\subsection{Invariant subspaces}
Let the functions $E \in \Phi(\tau,\er)$ and $H \in \Phi(\nu,\mu)$ be expanded
in the series \eqref{23}, \eqref{24}, \eqref{25}, \eqref{26}. 
Theorem \ref{t25} and Lemma \ref{l246} mean that the action of the Maxwell operator can be written in the following way:
\begin{eqnarray}
\label{31}
{\cal M} 
\left( \begin{array}{cc} \sum_{k=1}^\infty a_k \dd_1 \ph_k 
+ \sum_{l=2}^\infty b_l \dd_2 \psi_l
+\sum_{j=1}^{N-1} d_j \dd_1 \oo_j \\
 \sum_{k=1}^\infty a_k \dd_2 \ph_k 
- \sum_{l=2}^\infty b_l \dd_1 \psi_l
+\sum_{j=1}^{N-1} d_j \dd_2 \oo_j \\
\sum_{m=1}^\infty c_m \ph_m \\
\sum_{l=2}^\infty e_l \dd_1 \psi_l 
+ \sum_{k=1}^\infty f_k \dd_2 \ph_k 
+ \sum_{j=1}^{N-1} h_j \dd_2 \oo_j \\
\sum_{l=2}^\infty e_l \dd_2 \psi_l 
- \sum_{k=1}^\infty f_k \dd_1 \ph_k 
- \sum_{j=1}^{N-1} h_j \dd_1 \oo_j \\
\sum_{m=1}^\infty g_m \psi_m
\end{array} \right) \\
= \left( \begin{array}{cc}
i \er^{-1} \left( \sum_{k=1}^\infty f_k' \dd_1 \ph_k 
+ \sum_{l=2}^\infty (g_l - e_l') \dd_2 \psi_l
+\sum_{j=1}^{N-1} h_j' \dd_1 \oo_j\right) \\
i \er^{-1} \left( \sum_{k=1}^\infty f_k' \dd_2 \ph_k 
- \sum_{l=2}^\infty (g_l - e_l') \dd_1 \psi_l
+\sum_{j=1}^{N-1} h_j' \dd_2 \oo_j\right) \\
i \er^{-1} \sum_{k=1}^\infty \la_k f_k \ph_k \\
-i \mu^{-1} \left( \sum_{l=2}^\infty b_l' \dd_1 \psi_l 
+ \sum_{k=1}^\infty (c_k-a_k') \dd_2 \ph_k 
- \sum_{j=1}^{N-1} d_j' \dd_2 \oo_j\right) \\
-i \mu^{-1} \left( \sum_{l=2}^\infty b_l' \dd_2 \psi_l 
- \sum_{k=1}^\infty (c_k-a_k') \dd_1 \ph_k 
+ \sum_{j=1}^{N-1} d_j' \dd_1 \oo_j\right) \\
-i \mu^{-1} \sum_{l=2}^\infty \ka_l b_l \psi_l
\end{array} \right) .
\nonumber
\end{eqnarray}

Introduce the subspaces
$$
{\cal J}_k^{el} =
\left\{\left( \begin{array}{cc} a(x_3) \dd_1\ph_k (x_1, x_2) \\
a(x_3) \dd_2 \ph_k (x_1, x_2) \\
c(x_3) \ph_k (x_1, x_2) \\
f(x_3) \dd_2 \ph_k (x_1, x_2) \\
- f(x_3) \dd_1 \ph_k (x_1, x_2) \\
0\end{array} \right) , \quad
\begin{array}{cc} 
a, c, f \in L_2(\R) \ \text{such that} \\
\er c \in W_2^1(\R), (\er c)' = \la_k \er a
\end{array} \right\}, \quad k \in \N ,
$$
$$
{\cal J}_l^m = 
\left\{\left( \begin{array}{cc} b(x_3) \dd_2 \psi_l(x_1, x_2) \\
- b(x_3) \dd_1 \psi_l(x_1, x_2) \\
0 \\
e(x_3) \dd_1 \psi_l (x_1, x_2) \\
e(x_3) \dd_2 \psi_l (x_1, x_2) \\
g(x_3) \psi_l(x_1,x_2)
\end{array} \right), \quad
\begin{array}{cc} 
b, e, g \in L_2(\R) \ \text{such that} \\
\mu g \in W_2^1(\R), (\mu g)' = \ka_l \mu e
\end{array} \right\}, \quad l \ge 2,
$$
$$
{\cal J}_j^0 =
\left\{\left( \begin{array}{cc} d(x_3) \dd_1\oo_j (x_1, x_2) \\
d(x_3) \dd_2 \oo_j (x_1, x_2) \\
0 \\
h(x_3) \dd_2 \oo_j (x_1, x_2) \\
- h(x_3) \dd_1 \oo_j (x_1, x_2) \\
0\end{array} \right) , \quad
d, h \in L_2(\R) \right\}, \quad j = 1, \dots, N-1 .
$$

Lemma \ref{l21}, Corollary \ref{c22} and Lemma \ref{l23} yield the identity
$$
{\cal J} = \left(\bigoplus_{k=1}^\infty {\cal J}_k^{el}\right) \bigoplus
\left(\bigoplus_{l=2}^\infty {\cal J}_l^m\right) \bigoplus
\left(\bigoplus_{j=1}^{N-1} {\cal J}_j^0\right) .
$$
The formula \eqref{31} implies that this decomposition reduces the Maxwell operator.
It is natural to identify the spaces ${\cal J}_k^{el}$ with the spaces of vector-function
on the real line
\begin{equation}
\label{318}
\left\{\left( \begin{array}{cc} a\\ c\\ f \end{array} \right) \in L_2(\R, \C^3): (\er c)' = \la_k \er a\right\}
\end{equation}
with the norm
\begin{equation}
\label{319}
\int_\R \left(\la_k \er(z) |a(z)|^2 + \er(z) |c(z)|^2 + \la_k \mu(z) |f(z)|^2\right) dz .
\end{equation}
The part of the operator $\cal M$ in this subspace acts as
\begin{equation}
\label{32}
{\cal M}_k^{el}
= \left( \begin{array}{ccc}
0 & 0 & \frac{i}{\er(z)} \frac{d}{dz} \\
0 & 0 & \frac{i\la_k}{\er(z)} \\
\frac{i}{\mu(z)} \frac{d}{dz} & \frac{- i}{\mu(z)} & 0  \end{array} \right)
\end{equation}
on the domain
\begin{equation}
\label{33}
\dom {\cal M}_k^{el}
= \left\{\left( \begin{array}{cc} a\\ c\\ f \end{array} \right) \in L_2(\R, \C^3):
a, \er c, f \in W_2^1(\R), (\er c)' = \la_k \er a\right\}.
\end{equation}

The spaces ${\cal J}_l^m$ are naturally identified with the spaces of vector-function
on the real line
\begin{equation}
\label{333}
\left\{\left( \begin{array}{cc} b\\ e\\ g \end{array} \right) \in L_2(\R, \C^3): (\mu g)' = \ka_l \mu e\right\}
\end{equation}
with the norm
\begin{equation}
\label{336}
\int_\R \left(\ka_l \er(z) |b(z)|^2 + \ka_l \mu(z) |e(z)|^2 + \mu(z) |g(z)|^2\right) dz .
\end{equation}
The part of the operator $\cal M$ in this subspace acts as
\begin{equation}
\label{34}
{\cal M}_l^m
= \left( \begin{array}{ccc}
0 & - \frac{i}{\er(z)} \frac{d}{dz} & \frac{i}{\er(z)}\\
-\frac{i}{\mu(z)} \frac{d}{dz} & 0 & 0 \\
\frac{- i\ka_l}{\mu(z)} & 0 & 0 \end{array} \right)
\end{equation}
on the domain
\begin{equation}
\label{35}
\dom {\cal M}_l^m
= \left\{\left( \begin{array}{cc} b\\ e\\ g \end{array} \right) \in L_2(\R, \C^3):
b, e ,\mu g \in W_2^1(\R), (\mu g)' = \ka_l \mu e\right\}.
\end{equation}

Finally, we identify the spaces ${\cal J}_j^0$ with the space $L_2(\R, \C^2)$ with the norm
$$
\int_\R \left(\er(z) |d(z)|^2 + \mu(z) |h(z)|^2\right) dz .
$$
The part of the operator $\cal M$ in this subspace acts as
\begin{equation}
\label{36}
{\cal M}_j^0
= \left( \begin{array}{ccc}
0 & \frac{i}{\er(z)} \frac{d}{dz} \\
-\frac{i}{\mu(z)} \frac{d}{dz} & 0 \end{array} \right)
\end{equation}
on the domain
\begin{equation}
\label{37}
\dom {\cal M}_j^0 = W_2^1(\R,\C^2) .
\end{equation}

Thus, we have proved the

\begin{theorem}
\label{t31}
Let $U \subset \R^2$ be a bounded connected domain, $\dd U \in \lip$,
$\Pi = U \times \R$.
Let the coefficients  $\er$, $\mu$ be scalar real measurable functions satisfying \eqref{01} and \eqref{02}.
Then the Maxwell operator ${\cal M}$ is unitarily equivalent to the orthogonal sum 
of matrix differential operators of the first order on the real line
$$
\left(\bigoplus_{k=1}^\infty {\cal M}_k^{el}\right) \bigoplus
\left(\bigoplus_{l=2}^\infty {\cal M}_l^m\right) \bigoplus
\left(\bigoplus_{j=1}^{N-1} {\cal M}_j^0\right) ,
$$
where the operators ${\cal M}_k^{el}$, ${\cal M}_l^m$, ${\cal M}_j^0$
are defined by the formulas \eqref{32}, \eqref{33}, \eqref{34}, \eqref{35}, \eqref{36}, \eqref{37}
in the corresponding ${\cal J}$-spaces.
\end{theorem} 

\subsection{Proof of Theorem \ref{t13}}

\begin{lemma}
\label{l32}
Let ${\cal H}$ and ${\cal N}$ be the Hilbert spaces,
$R : {\cal H} \to {\cal N}$ be a closed operator, $\dom R$ be dense in ${\cal H}$,
$R^*: {\cal N} \to {\cal H}$ be its adjoint.
Assume that the kernels of these operators are trivial,
$\ker R = \{0\}$, $\ker R^* = \{0\}$.
Define the operator
$$
{\cal R} = \left( \begin{array}{cc}
0 & R^* \\
R & 0 \end{array} \right)
$$
in the Hilbert space ${\cal H} \oplus {\cal N}$ on the domain
$\dom {\cal R} = \dom R \oplus \dom R^*$.
Then ${\cal R}$ is unitarily equivalent to $-{\cal R}$, 
and its spectrum is symmetric with respect to zero.
Moreover, ${\cal R}$
is unitarily equivalent to the orthogonal sum of the operator 
$\sqrt{R^* R}$ in ${\cal H}$ and the operator $-\sqrt{R R^*}$ in ${\cal N}$.
\end{lemma}

\begin{proof}
Clearly, the operator ${\cal R}$ is self-adjoint.
The operators ${\cal R}$ and $-{\cal R}$ are unitarily equivalent due to the identity
$$
\left( \begin{array}{cc} 0 & R^* \\
R & 0 \end{array} \right)
\left( \begin{array}{cc}
1 & 0 \\
0 & -1 \end{array} \right)
= - 
\left( \begin{array}{cc}
1 & 0 \\
0 & -1 \end{array} \right)
\left( \begin{array}{cc}
0 & R^* \\
R & 0 \end{array} \right) .
$$
The square of the operator ${\cal R}$ is equal
$$
{\cal R}^2 = \left( \begin{array}{cc}
R^*R & 0\\
0 & RR^* \end{array} \right) .
$$
The operators $R^*R$ and $RR^*$ are unitarily equivalent under the assumption
$\ker R = \{0\}$, $\ker R^* = \{0\}$
(see \cite[Chapter 8, \S 1, Theorem 4]{BS}).
This implies the claim.
\end{proof}

\begin{rem}
In the same way, the  operator ${\cal R}$ is unitarily equivalent to the operator
$\sqrt{R^* R}\oplus - \sqrt{R^* R}$ in the space ${\cal H}$,
and to the operator $\sqrt{RR^*}\oplus - \sqrt{RR^*}$ in the space ${\cal N}$.
\end{rem}

\begin{rem}
Note also that Lemma \ref{l32} is valid without assumption of the triviality of the kernels of $R$ and $R^*$.
Indeed, in such situation one has to add to the operator ${\cal R}$ the zero operator $0$
in the subspace of dimension $(\dim\ker R + \dim \ker R^*)$,
to the operator $\sqrt{R^* R}$ the zero operator of dimension $\dim \ker R$,
and to the operator $-\sqrt{R R^*}$ the zero operator of dimension $\dim \ker R^*$.
\end{rem}

\begin{lemma}
\label{l33}
The operator ${\cal M}_k^{el}$ is unitarily equivalent to the orthogonal sum
$\sqrt{A_k^{el}} \oplus - \sqrt{A_k^{el}}$, where
$$
A_k^{el} = - \frac1{\mu(z)} \frac{d}{dz} \left(\frac1{\er(z)} \frac{d}{dz}\right) + \frac{\la_k}{\er(z)\mu(z)}
$$
is a self-adjoint operator in $L_2(\R,\mu dx)$,
$$
\dom A_k^{el} = \left\{ p \in W_2^1(\R) : (\er^{-1} p')' \in L_2(\R) \right\} .
$$
The operator $\left({\cal M}_k^{el}\right)^2$ is unitarily equivalent to the orthogonal sum of two operators $A_k^{el}$.
\end{lemma}

\begin{proof}
Let us consider the map
$$
\left( \begin{array}{cc} p \\ f \end{array} \right) \mapsto 
\left( \begin{array}{cc} a\\ c\\ f \end{array} \right) 
= \left( \begin{array}{cc} \la_k^{-1} \er^{-1} p' \\ \er^{-1} p\\ f \end{array} \right) .
$$
It is an isometric isomorphism between the space
$W_2^1(\R) \oplus L_2(\R)$
with the norm
$$
\int_\R \left(\er^{-1} \la_k^{-1} |p'|^2 + \er^{-1} |p|^2 + \mu \la_k |f|^2\right) dz
$$
and the space \eqref{318}, \eqref{319}.
This map transforms the operator ${\cal M}_k^{el}$ into the operator
$$
\tilde {\cal M}_k^{el}
= \left( \begin{array}{ccc}
0 & i\la_k \\
\frac{i}{\la_k\mu}\frac{d}{dz}\left(\frac1\er \frac{d}{dz}\right) - \frac{i}{\er\mu} & 0  \end{array} \right)
$$
on the domain
$$
\dom \tilde {\cal M}_k^{el}
= \left\{ p \in W_2^1(\R) : (\er^{-1} p')' \in L_2(\R) \right\} \oplus W_2^1(\R) .
$$
Now we apply Lemma \ref{l32} with 
$$
{\cal H} = W_2^1(\R)\qquad \text{with the norm} \quad \int_\R \er^{-1} \left(\la_k^{-1} |p'|^2 + |p|^2\right) dz;
$$
${\cal N} = L_2(\R, \la_k \mu dz)$;
$$
R = \frac{i}{\la_k\mu}\frac{d}{dz}\left(\frac1\er \frac{d}{dz}\right) - \frac{i}{\er\mu};
$$
$R^*$ is the operator of multiplication by $i\la_k$ acting from $L_2(\R)$ to $W_2^1(\R)$.
Clearly, $\ker R^* = \{0\}$.
It is also clear that the image $\ran R^*$ is dense in ${\cal H}$, and so $\ker R = \{0\}$.
Finally, $RR^* = A_k^{el}$.
Therefore, $\tilde {\cal M}_k^{el}$ is unitarily equivalent to $\sqrt{A_k^{el}} \oplus - \sqrt{A_k^{el}}$.
\end{proof}

\begin{lemma}
\label{l34}
The operator ${\cal M}_l^m$ is unitarily equivalent to the orthogonal sum
$\sqrt{A_l^m} \oplus - \sqrt{A_l^m}$, where
$$
A_l^m = - \frac1{\er(z)} \frac{d}{dz} \left(\frac1{\mu(z)} \frac{d}{dz}\right) + \frac{\ka_l}{\er(z)\mu(z)}
$$
is a self-adjoint operator in $L_2(\R,\er dx)$,
$$
\dom A_l^m = \left\{ q \in W_2^1(\R) : (\mu^{-1} q')' \in L_2(\R) \right\} .
$$
The operator $\left({\cal M}_l^m\right)^2$ is unitarily equivalent to the orthogonal sum of two operators $A_l^m$.
\end{lemma}

\begin{proof}
The map
$$
\left( \begin{array}{cc} b \\ q \end{array} \right) \mapsto 
\left( \begin{array}{cc} b\\ e\\ g \end{array} \right) 
= \left( \begin{array}{cc} b\\ \ka_l^{-1} \mu^{-1} q' \\ \mu^{-1} q \end{array} \right) .
$$
is an isometric isomorphism between the space
$L_2(\R) \oplus W_2^1(\R)$
with the norm
$$
\int_\R \left(\er \ka_l |b|^2 + \mu^{-1} \ka_l^{-1} |q'|^2 + \mu^{-1} |q|^2\right) dz
$$
and the space \eqref{333}, \eqref{336}.
This map transforms the operator ${\cal M}_l^m$ into the operator
$$
\tilde {\cal M}_l^m
= \left( \begin{array}{ccc}
0 & -\frac{i}{\ka_l\er}\frac{d}{dz}\left(\frac1\mu \frac{d}{dz}\right) + \frac{i}{\er\mu} \\
- i\ka_l  & 0  \end{array} \right)
$$
on the domain
$$
\dom \tilde {\cal M}_l^m
= W_2^1(\R) \oplus \left\{ q \in W_2^1(\R) : (\mu^{-1} q')' \in L_2(\R) \right\}  .
$$
Now we apply Lemma \ref{l32} with ${\cal H} = L_2(\R, \ka_l \er dz)$;
$$
{\cal N} = W_2^1(\R)\qquad \text{with the norm} \quad \int_\R \mu^{-1} \left(\ka_l^{-1} |q'|^2 + |q|^2\right) dz;
$$
$R$ is the operator of multiplication by $-i\ka_l$ acting from $L_2(\R)$ to $W_2^1(\R)$.
Clearly, $\ker R = \{0\}$, the image $\ran R$ is dense in ${\cal N}$, and $\ker R^* = \{0\}$.
Finally, $R^*R = A_l^m$.
Therefore, $\tilde {\cal M}_l^m$ is unitarily equivalent to $\sqrt{A_l^m} \oplus - \sqrt{A_l^m}$.
\end{proof}

\begin{lemma}
\label{l35}
The operator ${\cal M}_j^0$ is unitarily equivalent to the orthogonal sum
$\sqrt{A^0} \oplus - \sqrt{A^0}$, where
$$
A^0 = - \frac1{\mu(z)} \frac{d}{dz} \left(\frac1{\er(z)} \frac{d}{dz}\right) 
$$
is a self-adjoint operator in $L_2(\R,\mu dx)$,
$$
\dom A^0 = \left\{ p \in W_2^1(\R) : (\er^{-1} p')' \in L_2(\R) \right\} .
$$
The operator $\left({\cal M}_j^0\right)^2$ is unitarily equivalent to the orthogonal sum of two operators $A^0$.
\end{lemma}

\begin{proof}
Follows directly from Lemma \ref{l32}.
\end{proof}

\begin{rem}
\label{r48}
One could introduce the operator
$$
\tilde A^0 = - \frac1{\er(z)} \frac{d}{dz} \left(\frac1{\mu(z)} \frac{d}{dz}\right) 
$$
in the space $L_2(\R,\er dx)$,
$$
\dom \tilde A^0 = \left\{ q \in W_2^1(\R) : (\mu^{-1} q')' \in L_2(\R) \right\} 
$$
instead of the operator $A^0$.
These operators are unitarily equivalent.
\end{rem}

Now, Theorem \ref{t13} follows from Theorem \ref{t31} and
Lemmas \ref{l33}, \ref{l34} and \ref{l35}.

\section{Proof of Corollary \ref{c14}}
The operator $A_k^{el}$ corresponds to the quadratic form
$$
a_k^{el} [p] = \int_\R \er(z)^{-1} \left(|p'(z)|^2 + \la_k |p(z)|^2\right) dz,
\qquad \dom a_k^{el} = W_2^1(\R)
$$
in the space $L_2(\R,\mu dx)$.
Note that
$$
a_k^{el} [p] \ge \frac{\la_k}{\|\er\mu\|_{L_\infty(\R)}} \int_\R |p(z)|^2 \mu(z) \, dz ,
$$
so
\begin{equation}
\label{468}
A_k^{el} \ge \frac{\la_k}{\|\er\mu\|_{L_\infty(\R)}}\, I
\qquad \text{and} \qquad
\si (A_k^{el}) \subset  \left[\frac{\la_k}{\|\er\mu\|_{L_\infty}}, \infty\right).
\end{equation}
In the same way
\begin{equation}
\label{4681}
A_l^m \ge \frac{\ka_l}{\|\er\mu\|_{L_\infty(\R)}}\, I
\qquad \text{and} \qquad
\si (A_l^m) \subset  \left[\frac{\ka_l}{\|\er\mu\|_{L_\infty}}, \infty\right).
\end{equation}
Note also that the number $\frac{\la_k}{\|\er\mu\|_{L_\infty}}$ can not be an eigenvalue 
of the operator $A_k^{el}$.
Indeed, if
$$
A_k^{el} p = \frac{\la_k}{\|\er\mu\|_{L_\infty}} \, p,
$$
then
$$
\int_\R \er(z)^{-1} \left(|p'(z)|^2 + \la_k |p(z)|^2\right) dz
= \frac{\la_k}{\|\er\mu\|_{L_\infty}} \int_\R |p(z)|^2 \mu(z) \, dz 
$$
and therefore,
$$
\int_\R \er(z)^{-1} |p'(z)|^2 = 0
\qquad \Rightarrow \qquad  p \equiv \const 
\qquad \Rightarrow \qquad  p \equiv 0 .
$$
Thus,
\begin{equation}
\label{4685}
\si_p (A_k^{el}) \subset \left(\frac{\la_k}{\|\er\mu\|_{L_\infty}}, \infty\right).
\end{equation}
In the same way
\begin{equation}
\label{4686}
\si_p (A_l^m) \subset \left(\frac{\ka_l}{\|\er\mu\|_{L_\infty}}, \infty\right),
\end{equation}
\begin{equation}
\label{4687}
\si_p(A^0) \subset \left(0, \infty\right).
\end{equation}

Next, if the cross-section $U$ is simply connected then the operators $A^0$
in Theorem \ref{t13} are absent.
In such situation we need the inequality
\begin{equation}
\label{Fr}
\ka_2 < \la_1
\end{equation}
which holds true for arbitrary domains (see \cite{Fr} or \cite{F04}).
Now, \eqref{sp}, \eqref{468} and \eqref{4681} 
 imply
$$
\si\left({\cal M}^2\right) 
\subset \left[\frac{\ka_2}{\|\er\mu\|_{L_\infty}}, \infty\right).
$$
Therefore,
$$
\si\left({\cal M}\right) 
\subset \left(-\infty, - \sqrt{\frac{\ka_2}{\|\er\mu\|_{L_\infty}}}\right] \bigcup
\left[\sqrt{\frac{\ka_2}{\|\er\mu\|_{L_\infty}}}, \infty\right). \qquad \qed
$$

\section{Reduction to one-dimensional Schr\" odinger operator}
Now, it is natural to transform the operators $A_k^{el}$, $A_l^m$, $A^0$
into Schr\" odinger operators.
For this purpose we will assume that the coefficients are twice differentiable.
Such transformation is well known, we give the details for the reader convenience.

\subsection{Change of variables}

\begin{lemma}
\label{l41}
Let 
$$
\er \in W_{1,loc}^2 (\R) \cap L_\infty(\R)
\qquad \text{and} \qquad 
\sup_{s\in\R} \int_s^{s+1} |\er''(t)| \, dt =: S < \infty .
$$
Then $\er \in W_\infty^1 (\R)$.
\end{lemma}

\begin{proof}
Let $s\in\R$.
Then 
$$ \er'(s) = \int_{s-1}^s \left((t-s+1)\er'(t)\right)' dt 
= \int_{s-1}^s \left((t-s+1)\er''(t) + \er'(t)\right) dt,
$$
so
$$
|\er'(s)| \le \int_{s-1}^s \left|\er''(t)\right| dt + \left|\er(s) - \er(s-1)\right| \le  S + 2 \|\er\|_{L_\infty},
$$
and $\|\er'\|_{L_\infty} \le S + 2 \|\er\|_{L_\infty}$.
\end{proof}

It is convenient to change the variable by the rule
\begin{equation}
\label{41}
z \mapsto y (z) = \int_0^z \sqrt{\er(s) \mu(s)}\, ds .
\end{equation}
The function $y$ strictly increases, and it is a bijection $\R \to \R$.
Introduce the functions
\begin{equation}
\label{42}
\tilde \er(y) = \er(z), \qquad \tilde \mu(y) = \mu(z) .
\end{equation}
We have
\begin{equation}
\label{d}
dy = \sqrt{\er(z)\mu(z)} \, dz, 
\qquad dz = \frac{dy}{\sqrt{\tilde\er(y)\tilde\mu(y)}} .
\end{equation}

\begin{lemma}
\label{l43}
Let the functions $\er$, $\mu$ satisfy \eqref{01} and \eqref{03},
and the functions $\tilde\er$, $\tilde\mu$ be defined by \eqref{41}, \eqref{42}.
Then
$\tilde \er, \tilde \mu \in W_{1,loc}^2(\R)$, 
\begin{equation}
\label{45}
0 < \er_0 \le \tilde \er(y) \le \er_1 , \qquad 
0 < \mu_0 \le \tilde \mu(y) \le \mu_1 ,
\end{equation} 
and
\begin{equation}
\label{46}
\sup_{t\in\R} \int_t^{t+1} \left(|\tilde\er''(y)| + |\tilde\mu''(y)|\right) dy < \infty .
\end{equation}
\end{lemma}

\begin{proof}
Inequalities \eqref{45} are clear.
Let us prove \eqref{46}.
We have
\begin{equation}
\label{461}
\tilde \er'(y) = \frac{\er'(z)}{\sqrt{\er(z)\mu(z)}},
\end{equation}
\begin{equation}
\label{462}
\tilde \er''(y) = \frac{\er''(z)}{\er(z)\mu(z)}
- \frac{\er'(z)^2}{2\er(z)^2\mu(z)}
-\frac{\er'(z)\mu'(z)}{2\er(z)\mu(z)^2}.
\end{equation}
Therefore,
$$
\int_t^{t+1} |\tilde\er''(y)| \, dy 
\le C \int_{z(t)}^{z(t+1)} \left(|\er''(z)| + |\er'(z)|^2 + |\er'(z)\mu'(z)|\right) dz,
$$
where $z$ is the inverse function of the function $y$;
$$
z(t+1) - z(t) = \int_t^{t+1} z'(y)\, dy 
= \int_t^{t+1} \frac{dy}{\sqrt{\tilde\er(y)\tilde\mu(y)}} \le \frac1{\sqrt{\er_0\mu_0}}.
$$
Thus,
$$
\int_t^{t+1} |\tilde\er''(y)| \, dy \le 
C \left(\left[\frac1{\sqrt{\er_0\mu_0}}\right] + 1\right)
\left(\sup_{s\in\R} \int_s^{s+1} |\er''(z)| \, dz + \|\er'\|_{L_\infty}^2 + 
\|\er'\|_{L_\infty}\|\mu'\|_{L_\infty}\right).
$$
This together with Lemma \ref{l41} imply \eqref{46} for the function $\tilde \er$.
The argument for the function $\tilde \mu$ is just the same.
\end{proof}

\begin{lemma}
\label{l44}
Under the conditions of the precedent Lemma
introduce the functions
\begin{equation}
\label{423}
\nu (y) = \tilde \er (y)^{1/4} \tilde \mu(y)^{-1/4},
\end{equation}
\begin{equation}
\label{426}
\eta(y) = \frac{\nu'(y)}{\nu(y)}
= \frac14 \left(\frac{\tilde \er'(y)}{\tilde \er(y)} - \frac{\tilde \mu'(y)}{\tilde \mu(y)}\right),
\end{equation}
and
\begin{equation}
\label{43}
V(y) = \eta(y)^2 - \eta'(y) + \la \tilde\er(y)^{-1} \tilde\mu(y)^{-1}
\end{equation}
with some $\la \ge 0$.
Then
\begin{equation}
\label{47}
\nu \in W_{1,loc}^2 (\R) \cap L_\infty(\R) ,
\qquad \sup_{t\in\R} \int_t^{t+1} |\nu''(y)| \, dy < \infty ,
\end{equation}
\begin{equation}
\label{48}
\eta \in W_{1,loc}^1 (\R) \cap L_\infty(\R),
\qquad \sup_{t\in\R} \int_t^{t+1} |\eta'(y)| \, dy  < \infty ,
\end{equation}
\begin{equation}
\label{49}
V \in L_{1,loc}(\R), \qquad \sup_{t\in\R} \int_t^{t+1} |V(y)| \, dy  < \infty.
\end{equation}
\end{lemma}

\begin{proof}
The claim for $\nu$ follows from Lemma \ref{l43}.
Lemma \ref{l41} together with \eqref{47} yield $\nu' \in L_\infty(\R)$.
Moreover, $\nu$ is positive definite, therefore, $\eta \in L_\infty(\R)$.
Next,
$$
\eta'(y) = \frac{\nu''(y)}{\nu(y)} - \frac{\nu'(y)^2}{\nu(y)^2},
$$
which implies \eqref{48}.
Finally, \eqref{43} and \eqref{48} yield \eqref{49}.
\end{proof}

\subsection{Schr\" odinger operator}
Introduce the operator
$$
{\cal U} : p \mapsto u, \qquad
u(y) = \mu(z)^{1/4} \er(z)^{-1/4} p(z),
$$
where the variables $y$ and $z$ are related via \eqref{41}.
Then
\begin{equation}
\label{U}
\left({\cal U}^{-1} u\right) (z) = \nu(y) u(y),
\end{equation}
where the function $\nu$ is defined by \eqref{423}.

\begin{lemma}
\label{l45}
The operator
$$
{\cal U} : L_2(\R, \mu dz) \to L_2(\R)
$$
is an isometric isomorphism.
\end{lemma}

\begin{proof}
We have
$$
\int_\R |p(z)|^2 \mu(z) \, dz
= \int_\R \tilde \mu(y)^{-1/2} \tilde \er(y)^{1/2} |u(y)|^2 \tilde \mu(y) \frac{dy}{\sqrt{\tilde\er(y)\tilde\mu(y)}}
= \int_\R |u(y)|^2 dy .
\qquad \qedhere
$$
\end{proof}

Now, let us consider two quadratic forms.
The first one is
\begin{equation}
\label{495}
h[u] = \int_\R\left(|u'(y)|^2 + V(y) |u(y)|^2\right) dy, 
\qquad \dom h = W_2^1(\R)
\end{equation}
in the space $L_2(\R)$.
Here the function $V$ is defined by \eqref{43}.
It is well known (see, for example, \cite[Chapter 10, \S 6.1]{BS})
that under the condition \eqref{49} the form $h$ is closed and semi-bounded from below.
The corresponding self-adjoint operator is the Schr\" odinger operator 
$H = -\frac{d^2}{dy^2} + V(y)$.

The second quadratic form is 
$$
a [p] = \int_\R \er(z)^{-1} \left(|p'(z)|^2 + \la |p(z)|^2\right) dz,
\qquad \dom a = W_2^1(\R)
$$
in the space $L_2(\R,\mu dx)$.

\begin{lemma}
\label{l46}
Let the conditions of the Lemma \ref{l43} be fulfilled.
The operator ${\cal U}$ is a bijection of the space $W_2^1(\R)$ into itself, and
$$
a[p] = h[{\cal U} p] \qquad \forall \ p \in W_2^1(\R).
$$
\end{lemma}

\begin{proof}
The map 
${\cal U} : W_2^1(\R) \to W_2^1(\R)$ is a bijection 
because $\er', \tilde \er', \mu', \tilde \mu' \in L_\infty(\R)$.

Let $u\in W_2^1(\R)$, and $p = {\cal U}^{-1} u$.
By \eqref{U} and \eqref{d} we have
$$
p'(z) = \left(\nu'(y)u(y) + \nu(y)u'(y)\right) \sqrt{\tilde\er(y)\tilde\mu(y)},
$$
and
\begin{eqnarray*}
\int_\R \er(z)^{-1} |p'(z)|^2 dz =
\int_\R \tilde\er(y)^{-1} \left|\nu'(y)u(y) + \nu(y)u'(y)\right|^2 \sqrt{\tilde\er(y)\tilde\mu(y)} \, dy \\
= \int_\R \left(\nu(y)^2 |u'(y)|^2 + \nu'(y)\nu(y) \left(u'(y)\overline{u(y)} + u(y)\overline{u'(y)}\right)
+ \nu'(y)^2 |u(y)|^2\right) \tilde\er(y)^{-1/2} \tilde\mu(y)^{1/2} dy \\
=: I_1 + I_2 + I_3 .
\end{eqnarray*}
Using the definition \eqref{423} we get
$$
I_1 = \int_\R |u'(y)|^2 dy .
$$
The definition \eqref{426} gives us
$$
I_2 = \int_\R \eta (y) \frac{d}{dy} |u(y)|^2 dy 
= - \int_\R \eta'(y) |u(y)|^2 dy .
$$
Note that the integration by parts is correct here because the function $\eta$ is continuous and bounded 
on the real line due to Lemma \ref{l44},
and the function $u$ tends to zero at infinity, as $u \in W_2^1(\R)$.
For the last integral the formulas \eqref{423}, \eqref{426} yield
$$
I_3 = \int_\R \eta(y)^2 |u(y)|^2 dy,
$$
so
\begin{equation}
\label{410}
\int_\R \er(z)^{-1} |p'(z)|^2 dz =
\int_\R|u'(y)|^2 dy + \int_\R \left(\eta(y)^2 - \eta'(y)\right) |u(y)|^2 dy.
\end{equation}
Finally,
\begin{equation}
\label{411}
\int_\R \er(z)^{-1} |p(z)|^2 dz =
\int_\R \tilde\er(y)^{-1} \nu(y)^2 |u(y)|^2 \frac{dy}{\sqrt{\tilde\er(y)\tilde\mu(y)}}
= \int_\R \frac{|u(y)|^2 dy}{\tilde\er(y)\tilde\mu(y)}.
\end{equation}
The formulas \eqref{410}, \eqref{411} and \eqref{43} imply the equality
$a[p] = h[{\cal U} p]$.
\end{proof}

Lemmas \ref{l45} and \ref{l46} imply that the operator $A_k^{el}$ in Theorem \ref{t13}
is unitarily equivalent to the operator $H_k^{el} = -\frac{d^2}{dy^2} + V_k^{el} (y)$ in $L_2(\R)$,
where the potential $V_k^{el}$ is defined by the formula \eqref{43} with $\la = \la_k$.
The operator $A^0$ in Theorem \ref{t13} is the operator $A_k^{el}$ 
after the substitution $\la_k = 0$.
Therefore, $A^0$ is unitarily equivalent to the operator $H^0 = -\frac{d^2}{dy^2} + V^0 (y)$
with $V^0$ defined by the formula \eqref{43} with $\la = 0$.
The operator $A_l^m$ in Theorem \ref{t13} can be obtained from the operator $A_k^{el}$
by the changes $\er \leftrightarrow \mu$ and $\la_k \leftrightarrow \ka_l$.
Therefore, $A_l^m$ is unitarily equivalent to the operator $H_l^m = -\frac{d^2}{dy^2} + V_l^m (y)$,
the potential $V_l^m$ is obtained from $V_k^{el}$ by the same changes.
Now, Theorem \ref{t13} implies

\begin{theorem}
\label{t47}
Let $U \subset \R^2$ be a bounded connected domain, $\dd U \in \lip$,
the boundary $\dd U$ consists of $N$ connected components,
$\Pi = U \times \R$.
Let the coefficients  $\er$, $\mu$ be scalar real functions satisfying \eqref{01}, \eqref{02} and \eqref{03}.
Then the square ${\cal M}^2$ of the Maxwell operator is unitarily equivalent to the orthogonal sum
\begin{equation}
\label{616}
\left(\bigoplus_{k=1}^\infty H^{el}_k\right) \bigoplus \left(\bigoplus_{k=1}^\infty H^{el}_k\right) 
\bigoplus \left(\bigoplus_{l=2}^\infty H^m_l\right) \bigoplus \left(\bigoplus_{l=2}^\infty H^m_l\right)
\bigoplus \left(\bigoplus_{j=1}^{2N-2} H^0\right) .
\end{equation}
Here $H_k^{el}$, $H_l^m$, $H^0$ are self-adjoint operators in $L_2(\R)$,
$$
H_k^{el} = -\frac{d^2}{dy^2} + V_k^{el} (y), 
\qquad H_l^m = -\frac{d^2}{dy^2} + V_l^m (y),
\qquad H^0 = -\frac{d^2}{dy^2} + V^0 (y),
$$
\begin{equation}
\label{4e}
V_k^{el} (y) = \eta(y)^2 - \eta'(y) + \la_k \tilde\er(y)^{-1} \tilde\mu(y)^{-1},
\end{equation}
\begin{equation}
\label{4m}
V_l^m (y) = \eta(y)^2 + \eta'(y) + \ka_l \tilde\er(y)^{-1} \tilde\mu(y)^{-1},
\end{equation}
\begin{equation}
\label{40}
V^0 (y) = \eta(y)^2 - \eta'(y) ,
\end{equation}
the function $\eta$ is defined by \eqref{41}, \eqref{42}, \eqref{423} and \eqref{426}.
If the cross-section $U$ is simply connected ($N=1$), the summands $H^0$ in \eqref{616} are absent.
\end{theorem}

Theorem \ref{t47} and Lemma \ref{l32} yield
\begin{cor}
\label{c48}
Under the conditions of Theorem \ref{t47} the Maxwell operator ${\cal M}$ 
is unitarily equivalent to the orthogonal sum
\begin{eqnarray*}
\left(\bigoplus_{k=1}^\infty \sqrt{H^{el}_k}\right) \bigoplus \left(\bigoplus_{k=1}^\infty -\sqrt{H^{el}_k}\right) 
\bigoplus \left(\bigoplus_{l=2}^\infty \sqrt{H^m_l}\right) \bigoplus \left(\bigoplus_{l=2}^\infty -\sqrt{H^m_l}\right)\\
\bigoplus \left(\bigoplus_{j=1}^{N-1} \sqrt{H^0}\right) \bigoplus \left(\bigoplus_{j=1}^{N-1} -\sqrt{H^0}\right) .
\end{eqnarray*}
\end{cor}

\section{Coefficients tending to constants at infinity}
We will use the following well-known result (see for example \cite[Chapter 5]{Ya2}).

\begin{theorem}
\label{t51}
Let $W \in L_1(\R)$.
Let $H = -\frac{d^2}{dy^2} + W(y)$ be the self-adjoint operator in $L_2(\R)$
corresponding to the quadratic form
$$
h[u] = \int_\R\left(|u'(y)|^2 + W(y) |u(y)|^2\right) dy, 
\qquad \dom h = W_2^1(\R).
$$
Then
$$
\si_{ac}(H) = [0,+\infty), \qquad
\si_{sc}(H) = \emptyset \qquad
\text{and} \qquad \si_p(H) \cap (0,+\infty) = \emptyset .
$$
\end{theorem}

Note that negative eigenvalues can occur.

\begin{lemma}
\label{l52}
Let the coefficients $\er$, $\mu$ satisfy \eqref{01} and \eqref{02}.
Assume moreover that there are two constants $\er_*$, $\mu_*$ such that
\begin{equation}
\label{51}
\er - \er_* \in W_1^2 (\R), \qquad \mu - \mu_* \in W_1^2(\R) .
\end{equation}
Then
$$
\tilde \er - \er_* \in W_1^2 (\R), \qquad \tilde \mu - \mu_* \in W_1^2(\R) ,
$$
where the functions $\tilde\er$, $\tilde\mu$ are defined by \eqref{41}, \eqref{42}.
\end{lemma}

Clearly, the relations \eqref{51} imply that 
$\er_0\le\er_*\le\er_1$, $\mu_0\le\mu_*\le\mu_1$ and \eqref{03}.

\begin{proof}
The conditions $\er - \er_*, \mu - \mu_* \in L_1(\R)$ and the boundedness of the Jacobian of the map $z\mapsto y$ imply
$$
\tilde \er - \er_*, \tilde \mu - \mu_* \in L_1(\R) .
$$
The conditions $\er', \mu' \in L_1(\R)$ and the formula \eqref{461} imply $\tilde\er', \tilde\mu' \in L_1(\R)$.
Next, \eqref{51} yields $\er', \mu' \in L_\infty(\R)$, and therefore, $\er', \mu' \in L_2(\R)$.
Now, the inclusion $\tilde\er'', \tilde\mu'' \in L_1(\R)$ results from \eqref{51} and \eqref{462}.
So, 
$$
\tilde \er - \er_* \in W_1^2 (\R), \qquad \tilde \mu - \mu_* \in W_1^2(\R) . 
\qquad \qedhere
$$
\end{proof}

\begin{lemma}
\label{l53}
Let $\er$, $\mu$ satisfy \eqref{01}, \eqref{02} and \eqref{51}.
Then the potential $V$ defined by \eqref{43} satisfies
$$
V - \frac\la{\er_*\mu_*} \in L_1(\R).
$$
\end{lemma}

\begin{proof}
We have
\begin{equation*}
\frac1{\tilde\er\tilde\mu} - \frac1{\er_*\mu_*} \in L_1(\R).
\end{equation*}
By the precedent Lemma
$$
\frac{\tilde\er'}{\tilde\er}, \frac{\tilde\mu'}{\tilde\mu} \in W_1^1(\R) 
\qquad \Rightarrow \qquad \eta \in W_1^1(\R),
$$
where the function $\eta$ is defined by \eqref{426}.
Therefore, $\eta \in L_2(\R)$.
Now, the claim follows.
\end{proof}

\begin{cor}
Let $\er$, $\mu$ satisfy \eqref{01}, \eqref{02} and \eqref{51}.
Then the potentials defined by \eqref{4e}, \eqref{4m} and \eqref{40} satisfy the relations
$$
V_k^{el} - \frac{\la_k}{\er_*\mu_*}, \quad
V_l^m - \frac{\ka_l}{\er_*\mu_*}, \quad V^0 \in L_1(\R).
$$
\end{cor}

Theorem \ref{t51} and Lemma \ref{l53} mean that for the operator
\begin{equation*}
H = - \frac{d^2}{dy^2} + V(y)
\end{equation*}
defined via the quadratic form \eqref{495}
we have
\begin{equation}
\label{73}
\si_{sc}(H) = \emptyset, \quad
\si_{ac}(H) = \left[\frac\la{\er_*\mu_*}, + \infty\right),
\quad \si_p (H) \cap \left(\frac\la{\er_*\mu_*}, + \infty\right) = \emptyset .
\end{equation}
This together with the inequality \eqref{Fr} and Corollary \ref{c48} yield

\begin{cor}
\label{c54}
Let $\er$, $\mu$ satisfy \eqref{01}, \eqref{02} and \eqref{51}.
The singular continuous component in the spectrum of the Maxwell operator is absent,
$\si_{sc} ({\cal M}) = \emptyset$.
If the cross-section $U$ is multiply connected then 
$\si_{ac} ({\cal M}) = \R$.
If the cross-section $U$ is simply connected then
$$
\si_{ac} ({\cal M}) =  \left(-\infty, -\sqrt{\frac{\ka_2}{\er_*\mu_*}}\right] \cup 
\left[\sqrt{\frac{\ka_2}{\er_*\mu_*}}, +\infty\right) .
$$
\end{cor}

Let us study the point spectrum.

\begin{lemma}
Let the cross-section $U$ be multiply connected.
Let $\er$, $\mu$ satisfy \eqref{01}, \eqref{02} and \eqref{51}.
Then $\si_p (H^0) = \emptyset$, where the operator $H^0$ is defined in Theorem \ref{t47}.
\end{lemma}

\begin{proof}
Follows from \eqref{4687} and \eqref{73} with $\la = 0$.
\end{proof}

The existence of the eigenvalues of the operators $H_k^{el}$ and $H_l^m$ depends on the properties of the coefficients.

\begin{lemma}
\label{l55}
Let $\er$, $\mu$ satisfy \eqref{01}, \eqref{02}, \eqref{51}
and moreover,
$$
\er(z) \mu(z) \le \er_* \mu_* \qquad \forall \ z \in \R.
$$
Then 
$$\si_p (H_k^{el}) = \si_p (H_l^m) = \emptyset,
$$
where these operators are defined in Theorem \ref{t47}.
\end{lemma}

\begin{proof}
In this case 
$\|\er\mu\|_{L_\infty(\R)} = \er_* \mu_*$.
Now, the claim results from \eqref{4685}, \eqref{4686} and \eqref{73}.
\end{proof}

\begin{cor}
\label{c56}
Under the conditions of Lemma \ref{l55} the point spectrum of the Maxwell operator is empty, 
$\si_p ({\cal M}) = \emptyset$.
\end{cor}

\begin{lemma}
\label{l57}
Let $\er$, $\mu$ satisfy \eqref{01}, \eqref{02}, \eqref{51}
and moreover, there is $z_0 \in \R$ such that
$$
\er(z_0) \mu(z_0) > \er_* \mu_* .
$$
Then for $\la$ large enough
$$
\si_p(H) \cap \left(-\infty, \frac\la{\er_*\mu_*}\right) \neq \emptyset,
$$
where the operator $H$ is defined via \eqref{43}, \eqref{495}.
\end{lemma}

\begin{proof}
By assumption $\tilde\er(y_0) \tilde\mu(y_0) > \er_* \mu_*$,
where $y_0 = y(z_0)$.
By continuity, there are positive numbers $\de_1, \de_2 > 0$ such that
\begin{equation}
\label{53}
\frac1{\tilde\er(y) \tilde\mu(y)} \le \frac1{\er_*\mu_*} - \de_2 
\qquad \text{if} \quad |y-y_0| < 2 \de_1.
\end{equation}
Introduce the function $\ze \in W_2^1(\R)$,
\begin{equation}
\label{54}
\ze(y) = \begin{cases}
0, & y \le y_0 - 2 \de_1, \\
(y-y_0+2\de_1) \de_1^{-1}, & y_0 - 2 \de_1 < y < y_0 - \de_1, \\
1, & y_0 - \de_1 < y < y_0 + \de_1, \\
(y_0+2\de_1-y) \de_1^{-1}, &  y_0 + \de_1 < y < y_0 + 2\de_1, \\
0, & y\ge y_0 + 2\de_1 .
\end{cases}
\end{equation}
Then
$$
\int_\R\left( \frac1{\tilde\er(y) \tilde\mu(y)} - \frac1{\er_*\mu_*} \right) |\ze(y)|^2 dy
\le - \de_2 \int_\R |\ze(y)|^2 dy \le - 2 \de_1 \de_2 .
$$
Denote
\begin{equation}
\label{55}
\al : = \int_\R\left(|\ze'(y)|^2 + (\eta^2-\eta') |\ze(y)|^2\right) dy .
\end{equation}
We can estimate the value of the quadratic form \eqref{495} on the function $\ze$:
\begin{eqnarray*}
h[\ze] = 
\int_\R \left(|\ze'(y)|^2 + \left(\eta^2 - \eta' + \la \tilde\er^{-1} \tilde\mu^{-1}\right) \ze(y)|^2 \right) dy = 
\al + \la \int_\R \frac{|\ze(y)|^2 dy}{\tilde\er(y) \tilde\mu(y)}  \\
\le \al - 2\la \de_1 \de_2 + \frac\la{\er_*\mu_*} \int_\R |\ze(y)|^2 dy 
<  \frac\la{\er_*\mu_*} \int_\R |\ze(y)|^2 dy 
\end{eqnarray*}
if $\la > \al (2\de_1\de_2)^{-1}$.
For such $\la$, the last inequality and \eqref{73} guarantee the existence of eigenvalues
of the operator $H$ lesser than $\frac\la{\er_*\mu_*}$.
\end{proof}

\begin{cor}
\label{c58}
Under the conditions of Lemma \ref{l57}
the point spectrum of the Maxwell operator is infinite, $\# \si_p({\cal M}) = \infty$.
In particular, there are infinitely many eigenvalues on the continuous spectrum,
$$
\# \left(\si_p({\cal M}) \cap \si_{ac}({\cal M})\right)= \infty.
$$
\end{cor}

\begin{proof}
If $\la_k > \al (2\de_1\de_2)^{-1}$, where the numbers $\al, \de_1, \de_2$ are defined
by the formulas \eqref{53}, \eqref{54}, \eqref{55}, then the spectrum of the operator $H_k^{el}$
contains at least one eigenvalue.
The numbers $\la_k$ tend to infinity, so infinitely many operators $H_k^{el}$
has non-empty point spectrum.
The same is true for the operators $H_l^m$.
Now, the first claim follows from Corollary \ref{c48}.

If $\frac{\la_k}{\|\er\mu\|_{L_\infty}} \ge \frac{\ka_2}{\er_*\mu_*}$
then all the eigenvalues of $H_k^{el}$ are inside the absolute continuous spectrum of ${\cal M}^2$
by virtue of Corollary \ref{c54} and \eqref{468}.
Therefore, $\# \left(\si_p({\cal M}) \cap \si_{ac}({\cal M})\right)= \infty$.
\end{proof}

Theorem \ref{t15} follows from Corollaries \ref{c54}, \ref{c56} and \ref{c58}.

\section{Periodic coefficients}
We will use the following well-known result (see for example \cite{BES}).

\begin{theorem}
\label{t61}
Let $W$ be a real-valued periodic function,
$$
W \in L_{1,loc}(\R), \qquad W(y+b) = W(y).
$$
Let $H = -\frac{d^2}{dy^2} + W(y)$ be the self-adjoint operator in $L_2(\R)$
corresponding to the quadratic form
$$
h[u] = \int_\R\left(|u'(y)|^2 + W(y) |u(y)|^2\right) dy, 
\qquad \dom h = W_2^1(\R).
$$
Then

1) the spectrum of the operator is absolutely continuous,
$$
\si(H) = \si_{ac} (H), \qquad \si_{sc}(H) = \emptyset, \qquad \si_p (H) = \emptyset .
$$

2) The spectrum has a band-gap structure,
$$
\si(H) = \bigcup_{n=1}^\infty [\al_n, \be_n], \qquad \be_n \le \al_{n+1} .
$$
Here $\al_n$ (resp. $\be_n$) is the $n$-th eigenvalue of the operator $-\frac{d^2}{dy^2} + W(y)$
in $L_2(0,b)$ with periodic (resp. semiperiodic) boundary conditions if the number $n$ is odd,
and conversely, $\al_n$ (resp. $\be_n$) is the $n$-th eigenvalue of the operator $-\frac{d^2}{dy^2} + W(y)$
in $L_2(0,b)$ with semiperiodic (resp. periodic) boundary conditions if the number $n$ is even.
Moreover,
\begin{eqnarray*}
\al_n = \frac{\pi^2(n-1)^2}{b^2} + \frac1b \int_0^b W(y)\, dy + o(1), \quad n \to \infty,\\
\be_n = \frac{\pi^2 n^2}{b^2} + \frac1b \int_0^b W(y)\, dy + o(1), \quad n \to \infty.
\end{eqnarray*}
Some gaps can vanish (i.e. $\be_n = \al_{n+1}$ for some $n$), but in any cases
the lengths of gaps tend to zero, $\al_{n+1} - \be_n \to 0$.
\end{theorem}

{\it Proof of Theorem \ref{t16}.}
Coefficients $\er$, $\mu$ satisfy the conditions \eqref{01}, \eqref{02}, \eqref{03} and \eqref{04}.
The periodicity \eqref{04} implies that the function $y$ defined by \eqref{41} possess the property
$$
y(z+a)  = y(z) + \int_0^a \sqrt{\er(s)\mu(s)} ds.
$$
Therefore, the functions $\tilde\er$, $\tilde\mu$ are also periodic,
$$
\tilde\er(y+b) = \tilde\er(y), \quad \tilde\mu(y+b) = \tilde\mu(y),
\qquad \text{with} \quad b = \int_0^a \sqrt{\er(s)\mu(s)} ds.
$$
Moreover, the functions $\nu$, $\eta$ and $V$ defined by \eqref{423}, \eqref{426} and \eqref{43} 
are also periodic, and
$$
V_k^{el} (y+b) = V_k^{el} (y), \qquad V_l^m(y+b) = V_l^m(y), 
\qquad V^0(y+b) = V^0 (y).
$$
Thus, all the operators $H_k^{el}$, $H_l^m$, $H^0$ in Theorem \ref{t47} and Corollary \ref{c48} 
satisfy the conditions of Theorem \ref{t61}.
So, the spectra of the operators ${\cal M}^2$ and ${\cal M}$ are absolutely continuous.

Let us prove that the spectrum of ${\cal M}$ contains two symmetric semiaxes.
It is sufficient to prove that the set
$\si(H_1^{el}) \cup \si(H_2^{el})$ contains a semi-axes.
Recall that
$$
H_k^{el} =  -\frac{d^2}{dy^2} + V_k^{el} (y),
\qquad
V_k^{el} (y) = \eta(y)^2 - \eta'(y) + \la_k \tilde\er(y)^{-1} \tilde\mu(y)^{-1}.
$$
Denote by $w_k$ the mean value of the potential,
$$
w_k : = \frac1b \int_0^b V_k^{el} (y)\, dy.
$$
Clearly, $w_2 > w_1$ due to the inequality $\la_2 > \la_1$.
(Of course, we do not need to use just the inequality $\la_2 > \la_1$;
one could take any two different eigenvlaues $\la_j >\la_i$.)
Denote by $\al_n^{(k)}$, $\be_n^{(k)}$ the edges of the gaps in the spectra of the operators $H_k^{el}$,
$k=1,2$.
Let $0 < \de < (w_2-w_1)/2$.
Fix a natural number $N$ such that
\begin{equation}
\label{61}
(2N+1) \pi^2 > (w_2-w_1+2\de) b^2
\end{equation}
and
$$
\left|\al_{n+1}^{(k)} - \pi^2 n^2 b^{-2} - w_k\right| < \de,
\quad \left|\be_n^{(k)} - \pi^2 n^2 b^{-2} - w_k\right| < \de \qquad \text{if} \ \ k = 1,2, \ \ n \ge N.
$$
We show that 
\begin{equation}
\label{62}
[K, \infty) \subset \si(H_1^{el}) \cup \si(H_2^{el}),
\end{equation}
where
$$
K = \pi^2 N^2 b^{-2} + w_2 + \de .
$$
Indeed, if $\la \ge K$ and $\la \notin \si(H_1^{el})$ then
$$
\left|\la - \pi^2 n^2 b^{-2} - w_1\right| < \de \qquad \text{for some} \ \ n \ge N.
$$
If $\la \ge K$ and $\la \notin \si(H_2^{el})$ then
$$
\left|\la - \pi^2 m^2 b^{-2} - w_2\right| < \de \qquad \text{for some} \ \ m \ge N.
$$
Then
\begin{equation}
\label{63} 
\left|\pi^2 n^2 b^{-2} - \pi^2 m^2 b^{-2} + w_1 - w_2\right| < 2 \de
\end{equation}
and
$$
\pi^2 \left|n^2-m^2\right| b^{-2} < w_2 - w_1 + 2\de .
$$
The last inequality together with \eqref{61} yield $n=m$.
But the inequality \eqref{63} can not be valid if $n=m$ by the choice of $\de$.
We get a contradiction, so \eqref{62} is proven.
Thus,
$$
[K, \infty) \subset \si ({\cal M}^2) .
$$

If the number $\la_k$ (resp. $\ka_l$) is large enough then there are no spectrum of the operator $H_k^{el}$
(resp. $H_l^m$) below the point $K$, see \eqref{468} (resp. \eqref{4681}).
The sequences $\la_k$ and $\ka_l$ tend to infinity.
Therefore, by virtue of Theorem \ref{t47} the set of edges of gaps in the spectrum of ${\cal M}^2$ is finite.
The same is true for the operator ${\cal M}$.
\qed

\vskip1cm
St. Petersburg Department of Steklov Institute of Mathematics, 
191023, 27 Fontanka, St. Petersburg, 
and

St. Petersburg State University,
199034, Universitetskaya emb. 7/9, St.Petersburg, Russia.


\end{document}